\documentclass[12pt]{amsart}
\usepackage[all]{xy}
\usepackage{graphicx}
\usepackage{tikz}

\usepackage{amsmath}
\usepackage{amssymb}
\usepackage{amsfonts}
\usepackage{mathrsfs}
\usepackage{mathtools}

\newcommand{\Cdb}{\ensuremath{\mathbb{C}}}
\newcommand{\Ddb}{\ensuremath{\mathbb{D}}}

\newcommand{\Ndb}{\ensuremath{\mathbb{N}}}

\newcommand{\Pdb}{\ensuremath{\mathbb{P}}}

\newcommand{\Tdb}{\ensuremath{\mathbb{T}}}

\newcommand{\Zdb}{\ensuremath{\mathbb{Z}}}

\newcommand{\F}{\mbox{${\mathcal F}$}}

\renewcommand{\H}{\mbox{${\mathcal H}$}}

\renewcommand{\P}{\mbox{${\mathcal P}$}}

\newcommand{\R}{\mbox{${\mathcal R}$}}

\newcommand{\X}{\mbox{${\mathcal X}$}}

\newcommand{\norm}[1]{\Vert#1\Vert}
\newcommand{\bignorm}[1]{\bigl\Vert#1\bigr\Vert}
\newcommand{\Bignorm}[1]{\Bigl\Vert#1\Bigr\Vert}
\newcommand{\biggnorm}[1]{\biggl\Vert#1\biggl\Vert}
\newcommand{\Biggnorm}[1]{\Biggl\Vert#1\Biggr\Vert}

\newtheorem{theorem}{Theorem}[section]
\newtheorem{lemma}[theorem]{Lemma}
\newtheorem{corollary}[theorem]{Corollary}
\newtheorem{proposition}[theorem]{Proposition}
\newtheorem{definition}[theorem]{Definition}
\theoremstyle{remark}
\newtheorem{remark}[theorem]{\bf Remark}
\theoremstyle{definition}

\numberwithin{equation}{section}

\begin{document}

\title[$H^\infty$-functional calculus and isometric dilations]
{Connecting $H^\infty$-functional calculus and isometric dilations
for commuting families of Ritt$_E$ operators}

\author[C. Le Merdy]{Christian Le Merdy}
\email{clemerdy@univ-fcomte.fr}
\address{Laboratoire de Math\'ematiques de Besan\c con, 
Universit\'e de Franche-Comt\'e, 16 route de Gray
25030 Besan\c{c}on Cedex, FRANCE}

\author[M. N. Reshmi ]{M.N. Reshmi}
\email{rmazhava@math.univ-toulouse.fr}
\address{Institute de Math\'ematiques Toulouse,
Universit\'e  Paul Sabatier, 118, route de Narbonne
F-31062 Toulouse Cedex 9,   FRANCE}

\date{\today}

\begin{abstract} 
Let $(T_1,\ldots,T_d)$
be a commuting $d$-tuple of Ritt$_E$ operators on 
some UMD Banach space
$X$. We show that $(T_1,\ldots,T_d)$ admits a 
bounded $H^\infty$-functional calculus if and only if
$T_k$ is an $R$-Ritt$_E$ operator
for every $k=1,\ldots,d$, and the $d$-tuple $(T_1,\ldots,T_d)$ admits 
an isometric dilation $(U_1,\ldots,U_d)$ on some UMD Banach space $Y$
such that 
$(U_1,\ldots,U_d)$ is polynomially bounded. In the case where $X$ further possesses property $(\alpha)$, we establish other characterizations of the $H^\infty$-functional calculus property for $(T_1,\ldots,T_d)$ in terms of isometric dilations.
\end{abstract}

\maketitle



\vskip 1cm

\section{Introduction and main statements}\label{1Intro} 
In Banach space operator theory,
$H^\infty$-functional calculus 
occupies a significant position due to its applications to the
harmonic analysis of operators or $C_0$-semigroups, 
ergodic theory, multiplier theory
and  evolution problems. Throughout its development, 
isometric or isomorphic  dilations have always played a major role. 
The interactions between $H^\infty$-functional calculus and dilations 
are visible from the earliest works concerning sectorial operators. 
They were definitely highlighted by a remarkable paper
of A. Frohlich and L. Weis \cite{FW}, in which it is shown 
that if $X$ is a UMD Banach space and if $A$ is an 
$R$-sectorial operator of type $<\frac{\pi}{2}$ on $X$, then 
$A$ has a bounded $H^\infty$-functional calculus if and only if
the semigroup $(e^{-tA})_{t\geq 0}$ admits a dilation into a 
bounded $C_0$-group
on the Bochner space $L^2([0,1];X)$. Since then, connections 
between $H^\infty$-functional calculus and 
dilations for sectorial operators 
have been refined in various situations, and then 
it was studied for Ritt operators
in \cite{AFL, ALM}, see also \cite[Chapter 10]{LM-Book}.

$H^\infty$-functional calculus for commuting $d$-tuples of sectorial 
operators was introduced in \cite{A} and then investigated from
various perspectives
in \cite{FMI}, 
\cite{LLL} and \cite{KaW1}. The more recent paper \cite{ArLM2} studies 
dilations in this context. Finally, the study of
$H^\infty$-functional calculus for
commuting $d$-tuples of Ritt operators 
was addressed in \cite{Ar, ArLM1, MR} and 
dilation results of various kinds
were established in these papers.

In this work, we are interested in the class of Ritt$_E$ operators, 
also known as polygonal type operators.
Let  $\Tdb=\{z\in\Cdb\, :\, \vert z\vert=1\}$
denote the unit circle and let $\Ddb=\{z\in\Cdb\, :\, \vert z\vert<1\}$
denote the open unit disc.
Let $\xi_1,\ldots,\xi_N$ be distinct elements of $\Tdb$, for some $N\geq 1$, and let 
$E=\{\xi_1,\ldots,\xi_N\}$. Let $X$ be a Banach space. 
We say that a bounded operator $T\colon X\to X$
is a Ritt$_E$ operator if $\sigma(T)\subset\overline{\Ddb}$ and there exists a constant $C\geq 0$ such that
$$
\norm{(z-T)^{-1}}\leq C\max\bigl\{\vert \xi_j -z\vert^{-1}\, 
:\, j=1,\ldots,N\bigr\},\qquad 
z\notin \overline{\Ddb}.
$$
This resolvent estimate implies that $\sigma(T)\subset \Ddb\cup E$. 
An operator is called a polygonal type 
operator if it is Ritt$_E$ for some $E$. 
These operators generalize Ritt operators, indeed $T$ is a Ritt operator
if and only if it is Ritt$_E$ for the singleton $E=\{1\}$.
Ritt$_E$ operators first appeared in \cite{FMI,dL}, where it was proved that
if a polygonal type operator on Hilbert space is polynomially bounded,
then it is similar to a contraction. 
Quite recently,  polygonal type operators on Banach space and their 
$H^\infty$-functional calculus were studied in
\cite{B,BLM}, and several important results concerning Ritt operators were extended to this broader setting.
Moreover, \cite[Theorem 3.9]{BLM} provides an extension of the aforementioned result of
\cite{FMI,dL} to the Banach space setting. This extension requires the use of $R$-boundedness and
the notion of $R$-Ritt$_E$ operator, for which we refer to Sub-section \ref{2RRE}  below.

In the present paper, we investigate $H^\infty$-functional calculus for a commuting $d$-tuple of
Ritt$_E$ operators on some Banach space $X$, 
and its connection with isometric dilations. 
This topic was previously considered in 
\cite{LMRMN} in the case where $X$ is a Hilbert space, 
and in \cite{MPR}. Right after Theorem
\ref{1Main2}, we provide a precise comparison between 
our current work and the results in
\cite{LMRMN} and \cite{MPR}.

Concerning dilations, we will use the following terminology.

\begin{definition}\label{1Dilate}
Let $d\geq 1$ be an integer,  let $(T_1,\ldots,T_d)$ be a 
commuting $d$-tuple of 
operators on some Banach space $X$ and let $Y$ be another Banach space.
We say that  $(T_1,\ldots,T_d)$ admits an isometric dilation on $Y$ if there exist
two bounded maps $J\colon X\to Y$ and $Q\colon Y\to X$,  and commuting 
isometric isomorphisms $U_1,\ldots,U_d$ on $Y$ such that 
$$
T_1^{n_1}\cdots T_d^{n_d} =  Q U_1^{n_1}\cdots U_d^{n_d} J,\qquad n_1,\ldots,n_d\geq 0.
$$
Furthermore, if the $d$-tuple $(U_1,\ldots,U_d)$ is polynomially bounded,  then we say that 
$(T_1,\ldots,T_d)$ admits a polynomially bounded isometric dilation on $Y$.
\end{definition}

We refer to Definition \ref{2PB} for 
polynomial boundedness and to Sub-section \ref{2HFC} for the definition
of bounded $H^\infty$-functional calculus for a commuting $d$-tuple of
Ritt$_E$ operators. 
The main results of this paper are the following two theorems,
which will be proved in Sections \ref{4M1} and $\ref{5M2}$, respectively.
The geometric properties involved in these statements, 
namely UMD and property $(\alpha)$,
and their relevance in this context are discussed in 
Sub-section \ref{2Banach} and in 
Section \ref{5M2}.

\begin{theorem}\label{1Main1}
Let $X$ be a UMD Banach space, let $E\subset \Tdb$ be a finite set
and let $(T_1,\ldots,T_d)$
be a commuting $d$-tuple of Ritt$_E$ operators on $X$. The following assertions are equivalent.
\begin{itemize}
\item [(i)]  The $d$-tuple $(T_1,\ldots,T_d)$ admits a bounded $H^\infty$-functional calculus.
\item [(ii)]  For every $k=1,\ldots,d$, $T_k$ is an $R$-Ritt$_E$ operator, and $(T_1,\ldots,T_d)$ admits 
a polynomially bounded isometric dilation on some UMD Banach space $Y$.
\item [(iii)]   For every $k=1,\ldots,d$, $T_k$ is an $R$-Ritt$_E$ operator, and $(T_1,\ldots,T_d)$ 
admits a polynomially bounded isometric dilation on some Banach space $Y$.
\end{itemize}
\end{theorem}

\begin{theorem}\label{1Main2}
Let $X$ be a UMD Banach space with property $(\alpha)$, let $E\subset \Tdb$ be a finite set
and let $(T_1,\ldots,T_d)$
be a commuting $d$-tuple of Ritt$_E$ operators on $X$. 
The following assertions are equivalent.
\begin{itemize}
\item [(i)]  For every $k=1,\ldots,d$, $T_k$ admits a bounded $H^\infty$-functional calculus.
\item [(ii)]  The $d$-tuple $(T_1,\ldots,T_d)$ admits a bounded $H^\infty$-functional calculus.
\item [(iii)]  For every $k=1,\ldots,d$, $T_k$ is an $R$-Ritt$_E$ operator, and $(T_1,\ldots,T_d)$
admits a polynomially bounded isometric dilation on some UMD Banach space $Y$
with property $(\alpha)$.
\item [(iv)]  For every $k=1,\ldots,d$, $T_k$ is an $R$-Ritt$_E$ operator, and $(T_1,\ldots,T_d)$ 
admits an isometric dilation on some UMD Banach space $Y$.
\end{itemize}
\end{theorem}

In \cite{LMRMN}, we showed that if $X$ is a Hilbert space and if $T$ is 
a Ritt$_E$ operator on $X$ that admits a bounded $H^\infty$-functional 
calculus, then it admits an isometric dilation on a Hilbert space $Y$. 
Then, we used the specific form of this dilation to deduce that 
if a commuting $d$-tuple $(T_1,\ldots,T_d)$ of Ritt$_E$ 
operators on $X$ 
is such that each $T_k$ possesses a bounded $H^\infty$-functional
calculus, 
then $(T_1,\ldots,T_d)$ admits an isometric dilation on a 
Hilbert space $Y$. This approach was extended in \cite{MPR} to the 
case where $X$ is simply a reflexive Banach space such that $X$ 
and $X^*$ have finite cotype. Precisely, it is proven therein 
that if $(T_1,\ldots,T_d)$ is a commuting 
$d$-tuple of Ritt$_E$ operators 
on such a space $X$, and if each $T_k$ possesses a bounded 
$H^\infty$-functional calculus, then for any $1<p<\infty$,
$(T_1,\ldots,T_d)$ admits an isometric dilation on a Bochner
space $Y=L^p(\X;X)$.

This result and this approach are interesting in themselves,
but for commuting $d$-tuple $(T_1,\ldots,T_d)$ of Ritt$_E$ 
operators, they cannot lead to any characterization of the bounded 
$H^\infty$-functional calculus by a dilation property. Indeed,
it only takes into account the bounded $H^\infty$-functional calculus 
property for each individual $T_k$, which is generally a weaker property
than the bounded $H^\infty$-functional calculus property of the family 
$(T_1,\ldots, T_d)$, see Remark \ref{5CEX}. The main novelty of the work 
presented here is to consider the bounded $H^\infty$-functional calculus of a family
$(T_1,\ldots, T_d)$ (not just the bounded $H^\infty$-functional calculus of each of
its elements) and to introduce square functions associated with this 
property, see Section \ref{3SFE}. In the proof of Theorem \ref{1Main1},
see
Section \ref{4M1}, these square functions are precisely what allow for 
the construction of polynomially bounded isometric dilations, a property
that proves to be essential in establishing a characterization of the 
bounded $H^\infty$-functional calculus via dilations.

\section{Ritt$_E$-operators, functional calculus and Banach space geomerty}\label{2}

\subsection{Generalities}\label{Gen}
Throughout we let $X$ be a complex Banach space, we let $B(X)$ be the Banach algebra
of all bounded linear operators on $X$ and we let $I_X$ denote the identity operator 
on $X$. Sometimes, $I_X$ will be denoted by 1 for simplicity.
For any $T\in B(X)$, we let $\sigma(T)$ denote the spectrum of $T$. Then
for any $z\in\Cdb\setminus\sigma(T)$, we denote $R(z,T)=(z-T)^{-1}$ as the resolvent operator of $T$ at $z$.

For any $z_0\in \Cdb$ and any positive number $r>0$, we let $D(z_0,r)\subset\Cdb$ denote
the open disc centered at $z_0$ with radius $r$.

Let $d\geq 1$ be an integer. 
For any $\Omega\subset\Cdb^d$ and for any bounded function
$\varphi\colon\Omega\to \Cdb$, we set
$$
\norm{\varphi}_{\infty,\Omega} = \sup\bigl\{\vert \varphi(z)\vert\, :\, z\in\Omega\bigr\}.
$$
If further $\Omega$ is open, we let $H^\infty(\Omega)$ denote the space of all
bounded analytic functions
$\varphi\colon\Omega\to \Cdb$, equipped with $\norm{\,\cdotp}_{\infty,\Omega}$. This is a Banach algebra.

We let $\P_d$ denote the algebra of all complex polynomials in $d$ variables.

\begin{definition}\label{2PB} Let 
$(T_1,\ldots,T_d)$ be a commuting $d$-tuple of operators on $X$.
We say that $(T_1,\ldots,T_d)$ is polynomially bounded if there exists a constant $K\geq 1$ such that
$$
\norm{\varphi(T_1,\ldots,T_d)}\leq K\norm{\varphi}_{\infty, {\mathbb D}^d},\qquad \varphi\in\P_d.
$$
\end{definition}

\subsection{Ritt$_E$ operators and their functional calculus}\label{2HFC}
From now on, we fix a finite set  $E=\{\xi_1,\ldots,\xi_N\}\subset \Tdb$ of 
arbitrary size $N\geq 1$ (we assume that 
all the $\xi_j$ are distinct). Following \cite{BLM}, we say that 
an operator $T\in B(X)$ is a Ritt$_E$ operator if  $\sigma(T)\subset\overline{\Ddb}$ 
and there exists a constant $C\geq 0$ such that
$$
\norm{R(z,T)}\leq C\max\biggl\{\frac{1}{\vert \xi_j -z\vert}\,\, :\, j=1,\ldots,N\biggr\},\qquad
z\in \Cdb\setminus\overline{\Ddb}.
$$
It is easy to check (see \cite[Remark 2.3]{BLM}) that $T$ is Ritt$_E$ 
if and only if $\sigma(T)\subset\overline{\Ddb}$ and there exists a constant $C\geq 0$ such that
\begin{equation}\label{2Set-T}
\norm{R(z,T)}\leq \,\frac{C}{\prod_{j=1}^N\vert \xi_j -z \vert}\,,\qquad z\in D(0,2)\setminus\overline{\Ddb}.
\end{equation}
It was proved in \cite[Theorem 2.10]{BLM} that
$T$ is Ritt$_E$  if and only if $T$ is power bounded and  
there exists a constant $C_1\geq 0$ such that
\begin{equation}\label{2C1}
\Bignorm{T^{n-1}\prod_{j=1}^N(\xi_j-T)}\,\leq\,\frac{C_1}{n}\,,\qquad n\geq 1.
\end{equation}

Let $\rho\in(0,1)$. Following \cite[Definition 2.6]{BLM}, we let $E_\rho$ denote the interior 
of the convex hull of $E$ and the disc $D(0,\rho)$.
We say that a Ritt$_E$ operator $T$
is of type $\rho$  if $\sigma(T)\subset\overline{E_\rho}$
%$\rho$ is large enough in the sense of \cite[Remark 2.7]{BLM},
and for all $s\in(\rho,1)$, $E_s\cap\overline{E_\rho}=E$ and
there exists a constant $C\geq 0$ such that
\begin{equation}\label{2Strength}
\norm{R(z,T)}\leq \,\frac{C}{\prod_{j=1}^N\vert\xi_j-z\vert}\,,
\qquad z\in D(0,2)\setminus\overline{E_s}.
\end{equation} 
The existence of such a $\rho\in(0,1)$ is ensured by \cite[Lemma 2.8]{BLM}.

Note that in the case when $E=\{1\}$, $E_\rho$ coincides with the Stolz domain
$B_\rho$  considered in the study of Ritt operators
(see e.g. \cite[Figure 1]{LM-2014} or \cite[Paragraph 2.2]{LM-Book}).

We will need sectorial operators (in the bounded case), for which we refer either
to \cite{Ha} or \cite[Chapter 10]{HVVW2}. We recall 
from \cite[Lemma 2.4]{BLM} that if $T\in B(X)$
is Ritt$_E$, then for every $j=1,\ldots,N$, the operator
$I_X-\overline{\xi_j}T$ is sectorial. Therefore, we may consider the fractional powers 
$(I_X-\overline{\xi_j}T)^a$ for all $a>0$, see e.g. \cite[Chapter 3]{Ha}.
This will be used in (\ref{3Ts}) below.

For the remainder of this sub-section, we fix a commuting $d$-tuple $(T_1,\ldots,T_d)$
of Ritt$_E$ operators on $X$. 
For every $k=1,\ldots, d$, let $r_k\in (0,1)$ 
such that $T_k$ is of type $< r_k$. Then for any 
$u\in(0,1)$, we have
$\sigma(uT_1)\times\cdots\times\sigma(uT_d)\subset E_{r_1}\times\cdots\times E_{r_d}$. Hence
for any $\varphi$
in $H^\infty(E_{r_1}\times\cdots\times E_{r_d})$, we may define
$$
\varphi\bigl(uT_1,\ldots,uT_d)\,\in B(X),
$$
using the multi-variable Dunford-Riesz functional calculus.
More explicitly,
for any $v\in (u,1)$,
\begin{equation}\label{2Taylor}
\varphi(uT_1,\ldots,uT_d) = \Bigl(\frac{1}{2\pi i}\Bigr)^d
\int_{\prod_{k=1}^d \partial [vE_{r_k}]}
\varphi(z_1,\ldots,z_d) \prod_{k=1}^d R(z_k,uT_k)\,\prod_{k=1}^d dz_k,
\end{equation}
where  for every $k=1,\ldots, d$, the notation
$\partial [vE_{r_k}]$ stands for the boundary of $vE_{r_k}$ oriented counterclockwise.

\begin{definition}\label{2H} 

\

\begin{itemize}
\item [(1)] 
We say that $(T_1,\ldots,T_d)$ admits a bounded
$H^\infty(E_{r_1}\times\cdots\times E_{r_d})$-functional calculus if  
there exists a constant $K\geq 0$
such that
$$
\norm{\varphi(uT_1,\ldots,uT_d)}\leq K\norm{\varphi}_{\infty, 
E_{r_1}\times\cdots\times E_{r_d}},
$$
for all $\varphi\in H^\infty(E_{r_1}\times\cdots\times E_{r_d})$ and all $u\in (0,1).$

\smallskip
\item [(2)]
We say that $(T_1,\ldots,T_d)$ admits a bounded $H^\infty$-functional calculus if  
it admits a bounded
$H^\infty(E_{r_1}\times\cdots\times E_{r_d})$-functional calculus for some
$(r_1, \ldots,r_d)\in (0,1)^d$.
\end{itemize}
\end{definition}

The above definition, based on the approximation of 
$(T_1,\ldots,T_d)$ by
$(uT_1,\ldots,uT_d)$, looks different from 
classical definitions of bounded $H^\infty$-functional calculus in other contexts
(see in particular \cite{ArLM1, BLM, LM-2014}). 
However, we will see below how this definition 
aligns with those of these previous works.

\begin{lemma}\label{2Runge} 
The following assertions are equivalent.
\begin{itemize}
\item [(i)] The $d$-tuple $(T_1,\ldots,T_d)$ admits a bounded
$H^\infty(E_{r_1}\times\cdots\times E_{r_d})$ functional calculus.
\item [(ii)] There exists a constant $K\geq 1$ such that 
$$
\norm{\varphi(T_1,\ldots,T_d)}\leq K
\norm{\varphi}_{\infty, E_{r_1}\times\cdots\times E_{r_d}},\qquad \varphi\in\P_d.
$$
\end{itemize}
\end{lemma}

\begin{proof}
The implication ``$(i)\Rightarrow(ii)$" is clear, since for any polynomial $\varphi\in\P_d$,
we have a norm convergence 
$\varphi(uT_1,\ldots,uT_d)\to \varphi(T_1,\ldots,T_d)$, 
as $u\to 1$.

The converse implication ``$(ii)\Rightarrow(i)$" is a variant of the proof of \cite[Proposition 2.5]{ArLM1}, so we will be brief.
Given any  $\varphi\in H^\infty(E_{r_1}\times\cdots\times E_{r_d})$ and  any
$u\in(0,1)$, fix some $v\in(u,1)$ and apply Runge's theorem as stated in \cite[Lemma 2.4]{ArLM1}. We obtain
a sequence $(\varphi_m)_{m\geq 1}$ of $\P_d$ which converges uniformly to $\varphi$ on the closure of
$v E_{r_1}\times\cdots \times v E_{r_d}$. Applying (\ref{2Taylor}) to $\varphi_m$
and to $\varphi$, we deduce that 
$\varphi_m(uT_1,\ldots,uT_d)\to\varphi(uT_1,\ldots,uT_d)$, as $m\to\infty$. 
Moreover, the assumption (ii) applied to the polynomial  $(z_1,\ldots,z_d)\mapsto \varphi_m(uz_1,\ldots,uz_d)$
implies that 
$$
\norm{\varphi_m(uT_1,\ldots,uT_d)}\leq K\norm{\varphi_m}_{\infty,
v E_{r_1}\times\cdots \times v E_{r_d}},\qquad m\geq 1.
$$ 
Letting $m\to\infty$, this yields 
$\norm{\varphi(uT_1,\ldots,uT_d)}\leq K
\norm{\varphi}_{\infty, E_{r_1}\times\cdots\times E_{r_d}}$, which proves (i).
\end{proof}

The above lemma implies that a commuting $d$-tuple of Ritt$_E$ operators
with a bounded $H^\infty$-functional calculus is necessarily polynomially bounded.
See Remark \ref{4Converse} for a partial converse.

It follows from \cite[Proposition 2.5]{ArLM1} that 
in the context of Ritt operators (that is, when 
$E=\{1\}$), Definition \ref{2H}, (1) is equivalent to \cite[Definition 2.3]{ArLM1}.
Likewise, it follows from \cite[Proposition 3.4]{BLM} that if $d=1$, 
Definition \ref{2H}, (1) is equivalent to \cite[Definition 3.3]{BLM}.
Lemma \ref{2Approx} below provides another key result
linking our definition of a bounded 
$H^\infty$-functional calculus to \cite{ArLM1, BLM}.

For any $r_1,\ldots,r_d\in(0,1)$, we let
$H^\infty_0(E_{r_1}\times\cdots\times E_{r_d})$
be the space of all functions $\varphi$ in $H^\infty(E_{r_1}\times\cdots\times E_{r_d})$ for 
which there exist two positive real numbers  $c,\gamma>0$ such that
\begin{equation}\label{2H0}
\vert\varphi(z_1,\ldots,z_d)\vert\leq c\biggl(\prod_{k=1}^d\prod_{j=1}^N\vert \xi_j-z_k\vert\biggr)^\gamma,\qquad
z_1\in E_{r_1},\ldots, z_d\in E_{r_d}.
\end{equation}
This is an ideal of $H^\infty(E_{r_1}\times\cdots\times E_{r_d})$.

\begin{lemma}\label{2Approx}
Assume that for every $k=1,\ldots,d$, $T_k$ is of type $\rho_k$ for some $\rho_k\in(0,r_k)$
and let $s_k\in(\rho_k,r_k)$. Then, for all $\varphi\in H^\infty_0(E_{r_1}\times\cdots\times E_{r_d})$,
\begin{equation}\label{2App}
\varphi(uT_1,\ldots,uT_d)\longrightarrow
 \Bigl(\frac{1}{2\pi i}\Bigr)^d
\int_{\prod_{k=1}^d \partial E_{s_k}} 
\varphi(z_1,\ldots,z_d) \prod_{k=1}^d R(z_k,T_k)\,\prod_{k=1}^d dz_k
\end{equation}
as $u\to 1$, the integral in the right-hand side being absolutely convergent in $B(X)$.
\end{lemma}

\begin{proof}
For any $k=1,\ldots,d$, we apply \cite[Lemma 3.1]{B} to $T_k$. 
We obtain that there exists a constant 
$C \geq 0$ such that
\begin{equation}\label{2Uniform}
\norm{R(z,uT_k)}\leq \,\frac{C}{\prod_{j=1}^N\vert\xi_j-z\vert}\,,
\qquad z\in D(0,2)\setminus\overline{E_{s_k}},
\end{equation} 
for all $u\in (0,1)$. That is, the operators $uT_k$ satisfy (\ref{2Strength}) with $s=s_k$, uniformly in $u$.

Assuming (\ref{2H0}) and applying (\ref{2Uniform}), it is easy to check that 
the integral in the right-hand side of (\ref{2App}) is abolutely convergent. We let $S\in B(X)$ denote this integral.

Using Cauchy's theorem, we can rewrite (\ref{2Taylor}) as
$$
\varphi(uT_1,\ldots,uT_d) = \Bigl(\frac{1}{2\pi i}\Bigr)^d
\int_{\prod_{k=1}^d \partial E_{s_k}}
\varphi(z_1,\ldots,z_d) \prod_{k=1}^d R(z_k,uT_k)\,\prod_{k=1}^d dz_k.
$$
(We changed the contour of integration.)
Moreover (\ref{2H0}) and (\ref{2Uniform}) ensure  that  we have
\begin{align*}
\biggnorm{ \varphi(z_1,\ldots,z_d) \prod_{k=1}^d R(z_k,uT_k)} & \leq c
\biggl(\prod_{k=1}^d\prod_{j=1}^N\vert \xi_j-z_k\vert\biggr)^\gamma
\prod_{k=1}^d\norm{R(z_k,uT_k)}\\
& \leq\,\frac{c\,C}{\biggl(
\prod_{k=1}^d\prod_{j=1}^N\vert \xi_j-z_k\vert\biggr)^{1-\gamma}}\,,
\end{align*}
for all $(z_1,\ldots,z_d)\in \prod_{k=1}^d \partial E_{s_k}$ and all $u\in (0,1)$.

For any $k=1,\ldots,d$ and any $z_k\in \partial E_{s_k}\setminus E$, we have a norm convergence 
$R(z_k,uT_k)\to R(z_k,T)$, as $u\to 1$. Hence, applying the above estimate, we deduce from Lebesgue's dominated
convergence theorem that $\varphi(uT_1,\ldots,uT_d)\to S$, when $u\to 1$.
\end{proof}

Assume that  the hypotheses of Lemma \ref{2Approx} are satisfied. For
any  $\varphi\in H^\infty_0(E_{r_1}\times\cdots\times E_{r_d})$, we set 
$$
\varphi(T_1,\ldots,T_d) := 
\Bigl(\frac{1}{2\pi i}\Bigr)^d
\int_{\prod_{k=1}^d \partial E_{s_k}} 
\varphi(z_1,\ldots,z_d) \prod_{k=1}^d R(z_k,T_k)\,\prod_{k=1}^d dz_k.
$$
It is plain that  this definition 
does not depend on the choice of the $s_k\in (\rho_k,r_k)$.

We now give a few results 
that easily follow from
Lemma \ref{2Approx} 
and elementary properties of the multi-variable
Dunford-Riesz functional calculus.
They will be used silently in
the next sections. 

First, the 
mapping $\varphi\mapsto  \varphi(T_1,\ldots,T_d)$ is a
homomorphism from $H^\infty_0(E_{r_1}\times\cdots\times E_{r_d})$ into $B(X)$.
Moreover, we have
\begin{equation}\label{2Hom}
(\psi\varphi)(T_1,\ldots, T_d) = \psi(T_1,\ldots, T_d)
\varphi(T_1,\ldots, T_d),\qquad 
\varphi\in H^\infty_0(E_{r_1}\times\cdots\times E_{r_d}),\, 
\psi\in\P_d.
\end{equation}

Second, if $(T_1,\ldots,T_d)$ admits a 
bounded $H^\infty(E_{r_1}\times\cdots\times E_{r_d})$-functional calculus, then
the homomorphism $\varphi\mapsto\varphi(T_1,\ldots,T_d)$
is bounded with respect to the
$H^\infty(E_{r_1}\times\cdots\times E_{r_d})$-norm, that is, 
there exists a constant $K\geq 0$
such that
$$
\norm{\varphi(T_1,\ldots,T_d)}\leq K\norm{\varphi}_{\infty, 
E_{r_1}\times\cdots\times E_{r_d}},\qquad
\varphi\in H^\infty_0(E_{r_1}\times\cdots\times E_{r_d}).
$$

Third, 
consider a partition $\{1,\ldots,d\}
=\Gamma_1 \sqcup \Gamma_2$ into two non-empty subsets $\Gamma_1$ and $\Gamma_2$.
Let $\varphi_1\in H^\infty_0(\prod_{j\in\Gamma_1} E_{r_j})$ and $\varphi_2\in H^\infty_0(\prod_{j\in\Gamma_2} E_{r_j})$. We may define $\varphi\in H^\infty_0(E_{r_1}\times\cdots\times E_{r_d})$
by $\varphi(z_1,\ldots,z_d)=
\varphi_1((z_j)_{j\in\Gamma_1})
\varphi_2((z_j)_{j\in\Gamma_2})$. Then
$$
\varphi(T_1,\ldots,T_d)=
\varphi_1\bigl((T_j)_{j\in\Gamma_1}\bigr)
\varphi_2\bigl((T_j)_{j\in\Gamma_2}\bigr).
$$

Finally, in the case $d=1$,
let $r\in (0,1)$
and let $T$ be any Ritt$_E$ operator of type $<r$. Let $a>0$ 
and define $\varphi_{a}\in H^\infty_0(E_{r})$ by
$\varphi_{a}(z) =\prod_{j=1}^N(1-\overline{\xi_j}z)^{a}$. Then,
\begin{equation}\label{2Frac}
\varphi_{a}(T) = \prod_{j=1}^N(1-\overline{\xi_j}T)^{a}.
\end{equation}

\subsection{$R$-Ritt$_E$ operators}\label{2RRE}

For any $\sigma$-finite measure space $(\X,\mu)$ and any $1\leq p<\infty$,
we let $L^p(\X;X)$
denote the Bochner space of all measurable functions $\phi\colon\X\to X$ 
(defined up to almost everywhere zero functions) such that the norm function $t\mapsto\norm{\phi(t)}_X$
belongs to $L^p(\X)$. 
This is a Banach space for the norm $\norm{\phi}_{L^p({\mathcal X};X)} = 
\norm{\norm{\phi(\,\cdotp)}_X}_{L^p({\mathcal X})}$. 
We refer to  \cite[Chapters I-IV]{DU}  or \cite[Chapter I]{HVVW1}
for basic properties and information.

We will regard $L^p(\X) \otimes X$
as a subspace of $L^p(\X;X)$ in the usual way and we note that this subspace is dense.

Let $I$ be any countable set and let $(\varepsilon_i)_{i\in I}$ be a 
family of independent Rademacher variables on 
some probability space $(\X_0, \Pdb_0)$. We let ${\rm Rad}(I;X)$ denote the closure of the vector space 
${\rm Span}\{\varepsilon_i\otimes x\, :\, i\in I, x\in X\}$ in $L^2(\X_0;X)$ and we equip ${\rm Rad}(I;X)$
with the induced norm. Equivalently, ${\rm Rad}(I;X)$ is the closure of
${\rm Rad}_I\otimes X$ in $L^2(\X_0;X)$, where ${\rm Rad}_I
= {\rm Rad}(I;\Cdb)$.
When $I=\Ndb=\{1,2,\ldots\}$, we write ${\rm Rad}(X)$ instead of ${\rm Rad}(\Ndb;X)$
and we write ${\rm Rad}$ instead of ${\rm Rad}_{\mathbb N}$.
We refer to  \cite[Appendix A]{LM-Book} for basic information on these spaces.

A subset $\F\subset B(X)$ is called $R$-bounded if there exists a constant $C\geq 0$ such that
for all finitely supported sequences $(T_k)_{k\geq 1}$ of $\F$ and for all  finitely supported sequences
$(x_k)_{k\geq 1}$ of $X$,
$$
\Bignorm{\sum_k\varepsilon_k\otimes 
T_k(x_k)}_{{\rm Rad}(X)}\leq C
\Bignorm{\sum_k\varepsilon_k\otimes x_k}_{{\rm Rad}(X)}.
$$
In this case, we let $\R(\F)$ denote the smallest $C\geq 0$ which verifies this property.

According to (\ref{2Set-T}), the definition of Ritt$_E$ operators 
can be reformulated by saying that $T\in B(X)$ 
is Ritt$_E$  if $\sigma(T)\subset\overline{\Ddb}$ and the set
$$
\F_T : =\Bigl\{\Bigl(\prod_{j=1}^N(\xi_j-z)\Bigr)R(z,T)\, :\, z\in\Cdb,\ 1<\vert z\vert<2\Bigr\}
$$
is bounded. Following \cite[Definition 3.7]{BLM}, we say that $T$ is 
$R$-Ritt$_E$ if $\F_T$ is $R$-bounded.

\subsection{Background on Banach space geometry}\label{2Banach}
Results on $H^\infty$-functional calculus often require geometric assumptions on 
the underlying Banach space. This well-known phenomenon explicitly appears upon reading 
\cite{HVVW2} or \cite{LM-Book}.

In this context, UMD Banach
spaces were recognized as essential already in the seminal work of G. Dore and A. Venni
on bounded imaginary powers  \cite{DV}.
We refer to \cite{HVVW1, P5} for the definition and a comprehensive information 
on UMD Banach spaces. For readers unfamiliar with Banach space geometry,
we mention a few important properties: If $X$ is UMD, then
any Banach space isomorphic to $X$ is UMD;
If $X$ is UMD, then any subspace of $X$ is UMD;  
If $X$ is UMD, then its dual space $X^*$ 
is UMD.
All UMD Banach spaces are reflexive;
All Hilbert spaces, and all $L^p$-spaces, with $1<p<\infty$, are UMD; 
More generally, if $X$ is a
UMD Banach space, then any Bochner space $L^p(\X;X)$, with $1<p<\infty$,
is a UMD Banach space as well. 

We will use the notion of finite cotype, for which we refer to 
\cite{Maurey} (see also \cite[Chapter 7]{HVVW2}). We note that 
all UMD Banach spaces have finite cotype.

Let $Q\colon L^2(\X_0)\to L^2(\X_0)$ be the orthogonal projection onto ${\rm Rad}_I$.
We say that $X$ is a $K$-convex Banach space if the tensor extension
$Q\otimes I_X\colon  L^2(\X_0)\otimes X\longrightarrow  L^2(\X_0)\otimes X$
is bounded with respect to 
the $L^2(\X_0;X)$-norm. In this case, this tensor extension extends to a 
projection $L^p(\X_0;X)\to L^p(\X_0;X)$ whose range is equal to ${\rm Rad}(I;X)$. 
This property does not depend on $I$.
We refer to \cite[Chapter 13]{DJT} for general information on $K$-convexity, 
see also \cite[Chapter 5]{P5}, \cite{Maurey} and
\cite[Section 7.4]{HVVW2}. We will use the fact that UMD spaces are $K$-convex, 
see e.g. \cite[Proposition 4.3.10]{HVVW1}.

For any finitely supported families $(x_i)_{i\in I}$ and $(y_i)_{i\in I}$ in $X$ and $X^*$, respectively,
we may define
\begin{equation}\label{2Pairing}
\Bigl\langle \sum_{i\in I}\varepsilon_i\otimes y_i, \sum_{i\in I}
\varepsilon_i\otimes x_i \Bigr\rangle = \sum_{i\in I}\langle y_i,x_i\rangle,
\end{equation}
and we have
\begin{equation}\label{2CS}
\Bigl\vert \sum_{i\in I}\langle y_i,x_i\rangle\Bigr\vert\leq \Bignorm{\sum_{i\in I}\varepsilon_i\otimes x_i}_{{\rm Rad}(X)}
\Bignorm{\sum_{i\in I}\varepsilon_i\otimes y_i}_{{\rm Rad}(X^*)}.
\end{equation}
If $X$ is $K$-convex, then the duality pairing (\ref{2Pairing}) induces an isomorphic identification
\begin{equation}\label{2K}
{\rm Rad}(I;X)^*\approx {\rm Rad}(I;X^*).
\end{equation}
We refer e.g. to  \cite[Theorem 7.4.14, (2)]{HVVW2} for this result.

We set 
$$
{\rm Rad}^2(X)={\rm Rad}({\rm Rad}(X)).
$$
We say that $X$ has property $(\alpha)$
if there exists a constant $C\geq 1$ such that for all bounded families
$(z_{kj})_{k,j\geq 1}$ of $\Cdb$ and for all finitely supported families $(x_{kj})_{k,j\geq 1}$ 
of $X$, we have
\begin{equation}\label{2Alpha}
\Bignorm{\sum_{k, j} \varepsilon_k\otimes
\varepsilon_j \otimes z_{kj} x_{kj}}_{{\rm Rad}^2(X)}
\leq C \sup\bigl\{\vert z_{kj}\vert\, :\, k,j\geq 1\bigr\}
\Bignorm{\sum_{k,j} \varepsilon_k\otimes
\varepsilon_j \otimes x_{kj}}_{{\rm Rad}^2(X)}.
\end{equation}
We refer to \cite[Section 7.5]{HVVW2} for general
information on this notion (which is called Pisier's contraction
property there). We merely recall the following elementary facts:
If $X$ has property $(\alpha)$, then
any Banach space isomorphic to $X$ has property $(\alpha)$;
If $X$ has property $(\alpha)$, then any subspace of $X$ has property $(\alpha)$;  
All Hilbert spaces, and all $L^p$-spaces, with $1\leq p<\infty$, have property $(\alpha)$;
More generally, any Banach lattice with finite cotype has property $(\alpha)$.
We note that not all UMD Banach spaces have property $(\alpha)$. 
For example, for $1<p\not=2<\infty$,
non-commutative $L^p$-spaces are UMD but in general,
they do not have property $(\alpha)$.
See \cite[Appendices]{LM-Book} for details and references.

We say that $X$ has property $(\Delta)$ if it satisfies an estimate (\ref{2Alpha}) 
for the unique  family
$(z_{kj})_{k,j\geq 1}$ defined by $z_{kj}=1$ if $j\geq k$ and $z_{kj}=0$ if $j< k$.
In other words, the triangular projection is bounded on 
${\rm Rad}^2(X)$.
This notion is discussed in \cite[Subsection 7.5.b]{HVVW2}, under the name
of triangular contraction property. It is plain that any $X$ with property $(\alpha)$ has property $(\Delta)$.
A more subtle result due to N. Kalton and L. Weis asserts
that any UMD Banach space has property $(\Delta)$, see 
\cite[Proposition 3.2]{KaW1}  or \cite[Theorem 7.5.9]{HVVW2}.

The following was observed in \cite[Remark 3.12]{BLM}.

\begin{lemma}\label{2Delta} 
Let $T$ be a Ritt$_E$ operator on $X$ and 
assume that $X$ has property $(\Delta)$. If 
$T$ admits a bounded $H^\infty$-functional calculus, then $T$ is $R$-Ritt$_E$.
\end{lemma}

\section{Square function estimates}\label{3SFE}

Let $X$ be an arbitrary Banach space. Let $d\geq 1$ and let
$\alpha=(a_1,\ldots,a_d)\in(0,\infty)^d$ be a $d$-tuple of positive real numbers. 
Let $T=(T_1,\ldots,T_d)$ be a
commuting $d$-tuple of Ritt$_E$ operators. For any $x\in X$, we set
\begin{equation}\label{3Ts}
\norm{x}_{T,\alpha}
=\lim_{m}\,
\Biggnorm{\sum_{0\leq n_1,\ldots,n_d\leq m}
\varepsilon_{n_1,\ldots, n_d}\otimes \prod_{k=1}^d 
\Bigl((n_k+1)^{a_k-\frac12} T_k^{n_k}\prod_{j=1}^N
\bigl(I_X-\overline{\xi_j}T_k\bigr)^{a_k}\Bigr)x}_{{\rm Rad}
({\mathbb N}_0^d;X)}.
\end{equation}
The norm in the right-hand side is increasing with respect to $m$, 
hence the limit defining $\norm{x}_{T,\alpha}$
exists in $[0,\infty]$. We note that $\norm{x}_{T,\alpha}$ may be equal to $\infty$.

The above definition extends other prior definitions of square functions.
Indeed, in the case of Ritt operators (that is, when $E=\{1\}$),  (\ref{3Ts}) coincides with \cite[Eq. (2.8)]{Ar}, whereas in the case
of a single operator (that is, when $d=1$),  (\ref{3Ts}) coincides with \cite[Definition 2.5]{B}.

\begin{proposition}\label{3Implication}
Assume that $X$ has finite cotype. If
$T= (T_1,\ldots,T_d)$ admits a bounded $H^\infty$-functional calculus, then
there exists a constant $K\geq 0$ such that 
\begin{equation}\label{3Est}
\norm{x}_{T,\alpha}\leq K\norm{x},\qquad x\in X.
\end{equation}
\end{proposition}

\begin{proof} 
By assumption,  $(T_1,\ldots,T_d)$ admits a bounded 
$H^\infty(E_{\rho_1}\times\cdots\times E_{\rho_d})$-functional
calculus for some $\rho_1,\ldots,\rho_d$ in $(0,1)$. For every $k=1,\ldots,d$,
let $r_k\in(\rho_k,1)$. We claim that there exists a constant $K\geq 0$ such that
for any finite set $I$ and for all families $(\varphi_i)_{i\in I}$ of 
$H^\infty_{0}(E_{r_1}\times\cdots\times E_{r_d})$
and all $x\in X$, 
\begin{equation}\label{3Quadratic}
\Bignorm{\sum_{i\in I}\varepsilon_i\otimes \varphi_i(T_1,\ldots,T_d)x}_{{\rm Rad}(I;X)}
\leq K\norm{x}\Bignorm{\Bigl(\sum_{i\in I}\vert \varphi_i\vert^2\Bigr)^\frac12}_{\infty, 
E_{r_1}\times\cdots\times E_{r_d}}.
\end{equation}
This result is proved by O. Arrigoni
in \cite[Proposition 3.4]{Ar} for Ritt operators.
Arrigoni's proof relies on a Franks-McIntosh type decomposition adapted to Ritt operators, which was
established in \cite[Theorem 6.1]{ArLM1}, and on a factorization result 
observed in \cite[Remark 6.3]{ArLM1}. In \cite[Proposition 4.6 and Lemma 4.7]{B},
the aforementioned
decomposition result \cite[Theorem 6.1]{ArLM1} and factorization result 
\cite[Remark 6.3]{ArLM1} are extended to the context of Ritt$_E$ operators.
With these two results in hand, it is easy to adapt the proof of  \cite[Proposition 3.4]{Ar}
to Ritt$_E$ operators and to obtain (\ref{3Quadratic}). We skip the details.

Now we apply the above estimate with $I=\{0,\ldots,m\}^d$ and, 
for any $(n_1,\dots,n_d)\in I$,
$$
\varphi_{n_1,\ldots,n_d}(z_1,\ldots,z_d) = \prod_{k=1}^d 
(n_k+1) ^{a_k-\frac12} z_k^{n_k}\prod_{j=1}^N
(1-\overline{\xi_j}z_k)^{a_k},\qquad
z_1\in E_{r_1},\ldots, z_d\in E_{r_d}.
$$
These functions belong to $H^\infty_{0}(E_{r_1}\times\cdots\times E_{r_d})$ and by
(\ref{2Hom}) and (\ref{2Frac}), 
$$
\varphi_{n_1,\ldots,n_d}(T_1,\ldots,T_d) = \prod_{k=1}^d 
(n_k+1)^{a_k-\frac12} T_k^{n_k}\prod_{j=1}^N
\bigl(I_X-\overline{\xi_j}T_k\bigr)^{a_k}.
$$
To obtain the estimate (\ref{3Est}), it therefore  suffices to show that
\begin{equation}\label{3Sup}
\sup_{m\geq 1}\, 
\Bignorm{\Bigl(\sum_{0\leq n_1,\ldots,n_d\leq m}\vert \varphi_{n_1,\ldots,n_d}\vert^2\Bigr)^\frac12}_{\infty, 
E_{r_1}\times\cdots\times E_{r_d}}\,<\infty.
\end{equation}
To prove this, we introduce
$$
\psi_{n,a}(z) = (n+1)^{a-\frac12}z^{n} \prod_{j=1}^N
(1-\overline{\xi_j}z)^{a},\qquad
z\in\Ddb,
$$
for all $n\geq 0$ and $a\in(0,\infty)$. Recall the Stolz domains
$B_\beta$, $\beta\in\bigl(0,\frac{\pi}{2}\bigr)$, used
in the study of Ritt operators \cite[Figure 1]{LM-2014}. 
For any $r\in(0,1)$, there exist $\beta\in\bigl(0,\frac{\pi}{2}\bigr)$
and a covering
$$
E_r=D(0,r)\bigcup\bigl(\Omega_1\cup\cdots\cup\Omega_N\bigr)
$$
where, for each $j=1,\ldots,N$,
$\Omega_j$ is an open subset of $\xi_j B_\beta$.
For any $z\in E_r$, we have
\begin{equation}\label{3Sum}
\sum_{n=0}^\infty \vert \psi_{n,a}(z)\vert^2
=\prod_{j=1}^N \vert 1-\overline{\xi_j}z \vert^{2a}
\sum_{n=0}^\infty (n+1)^{2a-1}\vert z\vert^{2n}.
\end{equation}
It is clear that the right-hand side of the above identity
is uniformly bounded on $D(0,r)$. 
Let $1\leq j_0\leq N$. By \cite[Lemma 3.6]{Ar}, 
$$
\vert 1-\overline{\xi_{j_0}}z \vert^{2a}
\sum_{n=0}^\infty (n+1)^{2a-1}\vert z\vert^{2n}
$$
is uniformly bounded on $\Omega_{j_0}$, hence 
the right-hand side of (\ref{3Sum}) 
is uniformly bounded on $\Omega_{j_0}$. We deduce that
$$
\sup_{z\in E_r}\,\biggl(
\sum_{n=0}^\infty \vert \psi_{n,a}(z)\vert^2\biggr)\,<\infty.
$$
This implies (\ref{3Sup}), because for any $n_1,\ldots,n_d\geq 0$ and 
for any $z_1\in E_{r_1},\ldots,
z_d\in E_{r_d}$, we have
$$
\varphi_{n_1,\ldots,n_d}(z_1,\ldots,z_d)
=\prod_{k=1}^d \psi_{n_k,a_k}(z_k).
$$
\end{proof}

\begin{remark}\label{3CV}
Let $x\in X$ such that $\norm{x}_{T,\alpha}<\infty$. 
If $X$ does not contain $c_0$, that is, $X$ 
has no subspace isomorphic to $c_0$, then 
the family
$$
y_{n_1,\ldots,n_d} : = \varepsilon_{n_1,\ldots,n_d}\otimes 
\prod_{k=1}^d \biggl(
(n_k+1)^{a_k-\frac12} T_k^{n_k}\prod_{j=1}^N
(I_X-\overline{\xi_j}T_k)^{a_k}\biggr)x,\qquad (n_1,\ldots,n_d)\in \Ndb_0^{d},
$$
is summable in ${\rm Rad}(\Ndb_0^d;X)$. Indeed, the assumption
$\norm{x}_{T,\alpha}<\infty$ ensures that the finite sums of 
the $y_{n_1,\ldots,n_d}$ are uniformly bounded. Then, 
summability follows from  \cite{Kwa}.

Therefore, if $X$ has finite cotype and
$(T_1,\ldots,T_d)$ admits a bounded $H^\infty$-functional calculus, then for all 
$x\in X$, we obtain from Proposition \ref{3Implication} an element 
$$
\sum_{n_1,\ldots,n_d=0}^\infty
\varepsilon_{n_1,\ldots, n_d}\otimes \prod_{k=1}^d 
\biggl((n_k+1)^{a_k-\frac12} T_k^{n_k}\prod_{j=1}^N
\bigl(I_X-\overline{\xi_j}T_k\bigr)^{a_k}\biggr)x\,\in{\rm Rad}(\Ndb_0^{d};X),
$$
with norm $\leq K\norm{x}$, for some $K\geq 0$ not depending on $x$. Indeed, 
a Banach space with finite cotype does not contain $c_0$.
\end{remark}

We shall now use Gaussian spaces, which are defined in a parallel manner to the Rademacher
spaces.
Let $I$ be a countable set and let  $(g_i)_{i\in I}$ be a family of independent 
standard complex valued Gaussian variables on some probability space $(\X_1,\Pdb_1)$.
We let $G(I;X)\subset L^2(\X_1;X)$ be the closure
of ${\rm Span}\{g_i\otimes x\, :\, i\in I, x\in X\}$.

Assume that $X$ has finite cotype. Then the linear map
$$
{\rm Span}\{\varepsilon_i\otimes x\, :\, i\in I, x\in X\}
\longrightarrow  {\rm Span}\{g_i\otimes x\, :\, i\in I, x\in X\},
$$
which takes $\varepsilon_i\otimes x$ to $g_i\otimes x$ for all $i\in I$ and all $x\in X$,
extends to an isomorphic identification 
\begin{equation}\label{3RG}
{\rm Rad}(I;X)\approx G(I;X).
\end{equation}
We refer e.g. to \cite[Proposition 3.2 (ii)]{Pis85} or \cite[Proposition 12.27]{DJT} for this result.
(Note that this isomorphism holds only if $X$ has finite cotype, 
see \cite[Chapter 9]{LeTa}.)

With identification (\ref{3RG}) at our disposal,
the following is an immediate consequence of Proposition \ref{3Implication} and Remark \ref{3CV}.

\begin{corollary}\label{3Gaussian}
Assume that $X$ has finite cotype and that
$(T_1,\ldots,T_d)$ admits a bounded 
$H^\infty$-functional calculus. Then for any $x\in X$, 
the family
$$
g_{n_1,\ldots,n_d}\otimes \prod_{k=1}^d 
\biggl((n_k+1)^{a_k-\frac12} T_k^{n_k}\prod_{j=1}^N
\bigl(I_X-\overline{\xi_j}T_k\bigr)^{a_k}\biggr)x,\qquad 
(n_1,\ldots,n_d)\in \Ndb_0^{d},
$$
is summable in $G(\Ndb_0^d;X)$, and there exists a 
constant $K\geq 0$ such that 
$$
\Biggnorm{\sum_{n_1,\ldots,n_d =0}^\infty
g_{n_1,\ldots, n_d}\otimes \prod_{k=1}^d 
\biggl((n_k+1)^{a_k-\frac12} T_k^{n_k}\prod_{j=1}^N
(I_X-\overline{\xi_j}T_k)^{a_k}\biggr)
x\,}_{G({\mathbb N}_0^d;X)}\leq
K\norm{x},\qquad x\in X.
$$
\end{corollary}

Set $G_I=G(I;\Cdb)$, then $G_I\otimes X$ is a dense subspace of $G(I;X)$.
The following is a straightforward consequence of \cite[Corollary 12.17]{ DJT}.

\begin{lemma}\label{3Extension}
Let $V\colon G_I\to G_I$ be any bounded operator. 
Then the tensor extension $V\otimes I_X$ on $G_I\otimes X$ uniquely
extends to a bounded operator $V\overline{\otimes} I_X\colon G(I;X)\to G(I;X)$,
with 
$$
\norm{V \overline{\otimes} I_X}=\norm{V}.
$$
\end{lemma}

We note  in passing that the analogue of this lemma for 
${\rm Rad}(I;X)$ is false. This is the reason why we resort 
to Gaussian variables in our use of square functions.

\begin{remark}
To ``explain" why  $\norm{x}_{T,\alpha}$ in (\ref{3Ts})  is called a square function, we recall 
that if $X$ is a Banach lattice with finite cotype, then ${\rm Rad}(I;X)\approx X(\ell^2_I)$
for any countable set $I$; see e.g  \cite[Subsection 9.3.b]{HVVW2} or \cite[§A.13]{LM-Book}.
Therefore, in this case, there exist two constants $0<c<C$ such that for all $x\in X$,
$$
c\norm{x}_{T,\alpha}\leq
\biggnorm{\biggl(\sum_{n_1,\ldots,n_d =0}^\infty \prod_{k=1}^d 
(n_k+1)^{2a_k-1} \biggl\vert \prod_{k=1}^d 
T_k^{n_k}\prod_{j=1}^N
\bigl(I_X-\overline{\xi_j}T_k\bigr)^{a_k}x\biggr\vert^2\biggr)^\frac12}_X\leq C
\norm{x}_{T,\alpha}.
$$
Likewise, when $X=H$ is a Hilbert space, then ${\rm Rad}(I;H)= \ell^2_I(H)$ isometrically, hence
$$
\norm{x}_{T,\alpha} = \biggl(\sum_{n_1,\ldots,n_d = 0}^\infty\prod_{k=1}^d 
(n_k+1)^{2a_k-1}
\biggnorm{\prod_{k=1}^d  T_k^{n_k}\prod_{j=1}^N
\bigl(I_X-\overline{\xi_j}T_k\bigr)^{a_k}x}^2\biggr)^\frac12,\qquad x\in H.
$$
\end{remark}

\section{Proof of Theorem \ref{1Main1}}\label{4M1} 
We consider a commuting $d$-tuple $(T_1,\ldots,T_d)$ of Ritt$_E$-operators 
on some UMD Banach space $X$
and we wish to prove Theorem \ref{1Main1} for this family.
It is plain that (ii) implies (iii), so we only need to show that (i) 
implies (ii) and that (iii) implies (i).

Let  $(b_{n})_{n\geq 0}$ denote the sequence of complex numbers provided by the Taylor expansion,
\begin{equation}\label{taylor series expansion}
\frac{1}{\prod_{j=1}^{N}(1-\overline{\xi_j}z)}=\sum_{n=0}^{\infty}b_{n}z^{n},\qquad z\in\Ddb.
\end{equation}
By \cite[Lemma 3.2]{LMRMN}, this sequence is bounded. 

We set
$$
B := \prod_{k=1}^d\prod_{j=1}^N(I_X-~\overline{\xi_j}T_k).
$$
The following lemma is a sort of generalization of \cite[Lemma 3.3]{LMRMN}. 
It holds on any Banach space $X$.

\begin{lemma}\label{4identity Sk}
Let $(T_1,\ldots, T_d)$ be a $d$-tuple of
commuting Ritt$_E$ operators on 
$X$. For any $x\in{\rm Ran}(B),$
the family of $b_{n_1}\cdots b_{n_d}T_1^{n_1}\cdots T_d^{n_d}
Bx$, for $(n_1,\ldots,n_d)\in\Ndb_0^d$, is summable in $X$, and we have
$$
\sum_{n_1,\ldots,n_d=0}^\infty b_{n_1}\cdots b_{n_d}T_1^{n_1}\cdots T_d^{n_d}
Bx = x.
$$
\end{lemma}

\begin{proof}
If $T$ is a Ritt$_E$ operator,  then $\bignorm{T^{2n}\prod_{j=1}^{N}(\xi_j
-T)^2}\lesssim \frac{1}{n^2}$, by (\ref{2C1}). Hence, we have an estimate
$$
\Bignorm{T^{n}\prod_{j=1}^{N}(I_X-\overline{\xi_j} T)^2}\lesssim \frac{1}{n^2}.
$$
Thus, using the boundedness of $(b_n)_{n\geq 1}$,
we may define an operator 
$$
\sum_{n=0}^\infty b_n T^n \prod_{j=1}^{N}(I_X-\overline{\xi_j} T)^2\,\in B(X),
$$
the series being absolutely convergent.
By the Dunford-Riesz functional calculus and 
(\ref{taylor series expansion}), we have
$$
\sum_{n=0}^\infty b_n (uT)^n \prod_{j=1}^{N}(I_X-\overline{\xi_j} uT)^2 =  \prod_{j=1}^{N}(I_X-\overline{\xi_j} uT),
$$
for all $u\in(0,1)$. Letting $u\to 1$ and applying \cite[Lemma 3.2]{B}, we deduce that
$$
\sum_{n=0}^\infty b_n T^n \prod_{j=1}^{N}(I_X-
\overline{\xi_j}T)^2\,=\,
\prod_{j=1}^{N}(I_X-\overline{\xi_j} T).
$$
Since $T_1,T_2,\ldots,T_d$ are commuting 
Ritt$_E$ operators, we may apply the above to each $T_k$. 
Then, multiplying the resulting identities, we obtain
that the family of $b_{n_1}\cdots b_{n_d}T_1^{n_1}\cdots T_d^{n_d}
B^2$ is summable in $B(X)$, with
$$
\sum_{n_1,\ldots,n_d=0}^\infty b_{n_1}\cdots b_{n_d}T_1^{n_1}\cdots T_d^{n_d}
B^2 = B.
$$
The result follows at once.

% Thus by uniform boundedness of $S_k$, see (\ref{uniform bound of Sk}) we get $$S_k(T)x\rightarrow x$$ as $k\rightarrow\infty.$
% Since, $X$ has finite cotype and $T=(T_1,\ldots,T_d)$ admits a bounded $H^{\infty}-$ functional calculus, by Corollary \ref{3Gaussian}, with $a_k = \frac{1}{2},$ gives
% $$
% 		\Biggnorm{\sum_{n_1,\ldots,n_d =1}^\infty
% 			g_{n_1,\ldots, n_d}\otimes \prod_{k=1}^d  T_k^{n_k-1}\prod_{j=1}^N
% 			(I_X-\overline{\xi_j}T_k)^{\frac12}x\,}_{G({\mathbb N}^d;X)}\leq
% 		K\norm{x},\qquad x\in X.
% 		$$
%         If $X$ is of cotype $q$ for some $2\le q< \infty$,
%         $$\left(\sum_{n_1,\ldots,n_d =1}^\infty
% 			 \Biggnorm{\prod_{k=1}^d  T_k^{n_k-1}\prod_{j=1}^N
% 			(I_X-\overline{\xi_j}T_k)^{\frac12}x\,}_{X}^q\right)^{\frac1q}\le \Biggnorm{\sum_{n_1,\ldots,n_d =1}^\infty
% 			g_{n_1,\ldots, n_d}\otimes \prod_{k=1}^d  T_k^{n_k-1}\prod_{j=1}^N
% 			(I_X-\overline{\xi_j}T_k)^{\frac12}x\,}_{G({\mathbb N}^d;X)}.$$

%  Then, the same holds for all $x$ in the closure space, $\overline{\rm Ran}\biggl(\prod_{k=1}^d\prod_{j=1}^N (I-\overline{\xi_j}T_k)\biggr).$ 
    \end{proof}
% \begin{theorem}
%     Let $X$ be a reflexive Banach space and let T be a power bounded operator. Then
%     \begin{align}
%       X = Ker(IX − T ) ⊕ Ran(IX − T ).  
%     \end{align}
% Further if we let PT : X → X denote the projection with range Ran(PT ) = Ker(IX − T ) and
% kernel Ker(PT ) = Ran(IX − T ), then
% 
% 
% lim Mn (T ) (x) = PT (x)
% (1.21)
% n
% for all x ∈ X.
% \end{theorem}
% Let $x\in Ran(\prod_{i=1}^{d}\prod_{j=1}^{N}(I-\overline{\xi}_jT_i).$ Then, indeed $$x= Ran(\prod_{i=1}^{d}\prod_{j=1}^{N}(I-\overline{\xi}_jT_i)y$$, for some $y.$ Thus,
% \begin{align*}
%     \sum_{n_1,n_2,\ldots,n_d=0}^{\infty}\norm{b_{n_1}b_{n_2}\ldots b_{n_d}T_1^{n_1}&T^{n_2}\ldots T_d^{n_d}\prod_{i=1}^{d}\prod_{j=1}^{N}(I-\overline{\xi}_jT_i)x}\\=&\sum_{n_1,n_2,\ldots,n_d=0}^{\infty}\norm{b_{n_1}b_{n_2}\ldots b_{n_d}T_1^{n_1}T^{n_2}\ldots T_d^{n_d}\prod_{i=1}^{d}\prod_{j=1}^{N}(I-\overline{\xi}_jT_i)^2y}
%     <\infty.
% \end{align*}
% Thus, one may define unambiguously 
% $$\sum_{n_1,n_2,\ldots,n_d=0}^{\infty}b_{n_1}b_{n_2}\ldots b_{n_d}T_1^{n_1}T^{n_2}\ldots T_d^{n_d}\prod_{i=1}^{d}\prod_{j=1}^{N}(I-\overline{\xi}_jT_i)x.$$

\subsection{(i) implies (ii), special case}\label{4Special}
We assume that $(T_1,\ldots,T_d)$ admits a bounded
$H^\infty$-functional calculus. We further assume that 
\begin{equation}\label{Simple}
\overline{\rm Ran}\left(\prod_{k=1}^d\prod_{j=1}^N (I_X-\overline{\xi_j}T_k)\right)=X.
\end{equation}
The general case treated
in the next sub-section will rely on this special case.

The fact that every $T_k$ is an $R$-Ritt$_E$ operator follows from Lemma 
\ref{2Delta}.

We define
$$
A := \prod_{k=1}^{d}\prod_{j=1}^{N}(I_X-\overline{\xi_j}T_k)^{\frac{1}{2}},
$$
and we recall  $B=A^2$.
We set $\mathcal{C} = \{0,1\}^d$, and we let $(e_c)_{c\in \mathcal C}$ denote the standard 
Hilbertian basis of $\ell^2_{\mathcal C}$.
\iffalse
For each $c\in\mathcal{C},$ denote by $e_c$ the canonical basis vector in the space $l_{2^d}^p,$ with $e_0=e_{0,0,\ldots,0} =I_{l_{2^d}^p} $ for some $1<p<\infty.$
		Let $1<p'<\infty$ such that $$\frac{1}{p}+\frac{1}{p'}=1.$$ 
\fi
We define 
$$
J\colon X \longrightarrow \ell^2_{\mathcal C}\otimes 
G(\mathbb{Z}^d;X)
\qquad\hbox{and}\qquad
\tilde{J}\colon  X^* \longrightarrow \ell^2_{\mathcal C}\otimes G(\mathbb{Z}^d;X^*)
$$
as follows. 
For $x\in X$ and $y\in X^*$, 
\begin{align*}
J(x) & =\sum_{c\in\mathcal{C}}\sum_{m_1,\ldots,m_d=0}^\infty e_c\otimes  g_{m_1,\ldots, m_d}\otimes T^{m_1}_1\cdots T^{m_d}_dAx,\\
\tilde{J}(y) &=\sum_{c\in\mathcal{C}} \sum_{m_1,\ldots,m_d=0}^\infty
e_c\otimes g_{m_1,\ldots,m_d}
\otimes(b_{2m_1+c_1}\cdots b_{2m_d+c_d})
T^{*m_1}_1\cdots T^{*m_d}_d{A^*}(T^{c_1}_1\cdots T^{c_d}_d)^*y. 
\end{align*}
Being UMD, the Banach space $X$ has finite cotype.
Hence by Corollary \ref{3Gaussian}, applied with $a_k = \frac{1}{2}$ for each $k\in\{1,2,\ldots,d\}$, we see that $J$ 
is a well-defined bounded operator. Next,
$(T_1^*,\ldots,T_d^*)$ also admits a bounded
$H^\infty$-functional calculus and $X^*$ has finite cotype. 
Moreover the sequence $(b_n)_{n\geq 0}$ is bounded.
Therefore, $\tilde{J}$ is also a well-defined bounded operator.

For each $k\in\{1,2,\ldots,d\}$, define the operator 
$U_k\colon \ell^2_{\mathcal C} \otimes G(\mathbb{Z}^d;X)
\rightarrow \ell^2_{\mathcal C} \otimes G(\mathbb{Z}^d;X)$ by 
\begin{align*}
U_k\biggl(\sum_{c\in\mathcal{C}}\sum_{(m_1,\dots,m_d)\in{\mathbb Z}^d} 
e_c\otimes g_{m_1,\ldots,m_d} \otimes & 
x(c)_{m_1,\ldots,m_d} \biggr)= \\ &
\sum_{c\in\mathcal{C}}\sum_{(m_1,\dots,m_d)\in{\mathbb Z}^d} 
e_c\otimes g_{m_1,\ldots,m_d}\otimes x(c)_{m_1,\ldots,m_k
+1,\ldots,m_d}.
\end{align*}
In other words, each $U_k$ is a shift operator,
that shifts the $k$-th index of each term by one unit.
Clearly, the operators $U_1,U_2,\ldots,U_d$ are well defined and 
commute with each other. Let us equip 
$\ell^2_{\mathcal C} \otimes G(\mathbb{Z}^d;X)$ 
with the norm
$$
\Bignorm{\sum_{c\in{\mathcal C}} e_c\otimes
\gamma(c)} = \Bigl(\sum_{c\in\mathcal C}
\bignorm{\gamma(c)}_{G(\mathbb{Z}^d;X)}^2\Bigr)^\frac12,\qquad
\gamma(c)\in G(\mathbb{Z}^d;X),
$$
that is, we regard it as
$$
Y : = \ell^2_{\mathcal C}\bigl(G(\Zdb^d;X)\bigr).
$$
Then each $U_k$ is an isometric isomorphism on $Y$.

Let $V_1,\ldots,V_d$ be the shift operators on
$G_{{\mathbb Z}^d}$ defined by 
$$
V_k\biggl(\sum_{(m_1,\dots,m_d)\in{\mathbb Z}^d} 
t_{m_1,\ldots,m_d} \,g_{m_1,\ldots,m_d} \biggr)=
\sum_{(m_1,\dots,m_d)\in{\mathbb Z}^d} 
t_{m_1,\ldots, m_{k}+1,\ldots,m_d} \,g_{m_1,\ldots,m_d}.
$$
(Here, $t_{m_1,\ldots,m_d}\in\Cdb$.)
By the Spectral theorem for commuting normal operators,
the $d$-tuple $(V_1,\ldots,V_d)$ is polynomially bounded. More
precisely,
$$
\norm{\varphi(V_1,\ldots,V_d)}\leq \norm{\varphi}_{\infty,{\mathbb D}^d},\qquad \varphi\in\P_d.
$$
It is plain that for any $k$, $U_k$ is the unique extension of $I_{\ell^2_{\mathcal C}}\otimes 
V_k\otimes I_X$. Hence for all $\varphi\in\P_d$, 
$\varphi(U_1,\ldots,U_d)$ is the unique extension of
$I_{\ell^2_{\mathcal C}}\otimes 
\varphi(V_1,\ldots,V_d)\otimes I_X$. Applying Lemma \ref{3Extension},
we deduce that $(U_1,\ldots,U_d)$ is polynomially bounded.

Let $x\in X$. 
For any $k\in\{1,2,\ldots,d\}$ and 
any $n\geq 0$, we have 
$$
U_k^nJx = \sum_{c\in\mathcal{C}}\sum_{m_1,\ldots,m_d=0}^\infty
e_c\otimes  g_{m_1,\ldots, m_d}\otimes T^{m_1}_1\cdots T_k^{m_k+n}\cdots T^{m_d}_dA x.
$$
Hence for any $n_1,\ldots, n_k\geq 0$, we have
$$
U_1^{n_1}\cdots U_d^{n_d}J x = 
\sum_{c\in\mathcal{C}}\sum_{m_1,\ldots,m_d=0}^\infty
e_c\otimes  
g_{m_1,\ldots,m_d}\otimes T^{m_1+n_1}_1\cdots T^{m_k+n_k}_k
\cdots T^{m_d+n_d}_dAx.
$$

Being UMD, the Banach space $X$ is $K$-convex and reflexive. Hence
according to (\ref{2K}) and (\ref{3RG}),
we may regard $G(\mathbb{Z}^d;X)$
as the dual space of $G(\mathbb{Z}^d;X^*)$ (isomorphically). 
Therefore, we may regard $\ell^2_{\mathcal C}\otimes
G(\mathbb{Z}^d;X)$
as the dual space of $\ell^2_{\mathcal C}\otimes G(\mathbb{Z}^d;X^*)$.
This allows to consider the adjoint operator $\tilde{J}^*\colon 
\ell^2_{\mathcal C}\otimes G(\mathbb{Z}^d;X)\to X$. Then,  
for any $x\in{\rm Ran}(B)$ and any $y\in X^*$, we have
\begin{align*}
\bigl\langle y, &\tilde{J}^*U_1^{n_1}\cdots U_d^{n_d}J x\bigr\rangle
\\ &
= \langle \tilde{J}(y),U_1^{k_1}\ldots U_d^{k_d}J x\rangle \\
&=\sum_{c\in\mathcal C}
\sum_{m_1,\ldots,m_d=0}^\infty
(b_{2m_1+c_1}\cdots b_{2m_d+c_d})\bigl\langle
T_1^{*(m_1+c_1)}\ldots T_d^{*(m_d+c_d)} A^*y, T_1^{m_1+n_1}\cdots T_d^{m_d+n_d}
Ax\bigr\rangle\\
&=\sum_{c\in\mathcal C}
\sum_{m_1,\ldots,m_d=0}^\infty
(b_{2m_1+c_1}\cdots b_{2m_d+c_d})\bigl\langle
y, T_1^{2m_1 +c_1+n_1}\cdots T_d^{2m_d +c_d+n_d}Bx\bigr\rangle\\
& =\sum_{p_1,\ldots,p_d=0}^\infty (b_{p_1}\cdots b_{p_d}) \bigl\langle y, T_1^{p_1+n_1}\cdots T_d^{p_d+n_d}Bx\bigr\rangle\\
&= \sum_{p_1,\ldots,p_d=0}^\infty (b_{p_1}\cdots b_{p_d}) \bigl\langle y, T_1^{p_1}\cdots T_d^{p_d}B
\bigl(T_1^{n_1}\cdots T_d^{n_d}x\bigr) \bigr\rangle\\
&=\bigl\langle y, T^{n_1}_1\cdots T^{n_d}_d x\bigr\rangle,
\end{align*}
by Lemma \ref{4identity Sk}, since $T_1^{n_1}\cdots T_d^{n_d}x$ belongs to ${\rm Ran}(B)$.
Thus, for any $(n_1,n_2,\ldots,n_d)\in\mathbb{N}_0^d$, we have
\begin{align}\label{dilation of Ti}
T^{n_1}_1\cdots T^{n_d}_d(x) = \tilde{J}^*U^{n_1}_1\cdots U^{n_d}_d J(x),
\end{align}
for every $x\in {\rm Ran}(B).$
We assumed, see (\ref{Simple}), 
that ${\rm Ran}(B)$ is a dense subspace of $X$. 
Therefore, (\ref{dilation of Ti}) holds for every $x\in X$ and 
hence, $(U_1,U_2,\ldots,U_d)$ is  an isometric dilation of 
$(T_1,T_2,\ldots,T_d)$ on $Y$.
Since $G(\Zdb^d;X)$ is a subspace of $L^2(\X_1;X)$, this is a UMD Banach space.
Consequently, $Y$ is a UMD Banach space and therefore,
(ii) is satisfied.

\subsection{(i) implies (ii), general  case}
We will use a direct sum decomposition and the following two 
elementary lemmas, valid on any Banach space $X$.

\begin{lemma}\label{4Invariant}
Assume that $X=Z\oplus Z'$ for two closed subspaces
$Z,Z'\subset X$. Let $T\in B(X)$ such that $Z$ and $Z'$ are $T$-invariant. Then
$$
\overline{\rm Ran}\bigl(T_{\vert Z\to Z}\bigr) = \overline{\rm Ran}(T)\cap Z.
$$   
\end{lemma}

\begin{proof}
The inclusion $\subset$ is obvious. Conversely, 
let $x\in \overline{\rm Ran}(T)\cap Z$ and let 
$(y_n)_{n\geq 1}$ be a sequence of $X$ such that $T(y_n)\to x$. For any $n\geq 1$, let
$z_n \in Z$ and $z'_n\in Z'$ be the unique elements such that $y_n=z_n+z'_n$. Then 
$T(z_n)\in Z$ and $T(z'_n)\in Z'$ and therefore, $T(z_n)\to x$. Thus,
$x$ belongs to $\overline{\rm Ran}\bigl(T_{\vert Z\to Z}\bigr)$.
\end{proof}

\begin{lemma}\label{4DirectSumDilation}
Assume that $X=Z\oplus Z'$ for two closed subspaces
$Z,Z'\subset X$. Let $(T_1,\ldots,T_d)$ be a commuting $d$-tuple of operators on $X$
such that $Z$ and $Z'$ are $T_k$-invariant for every $k=1,\ldots,d$.
Set $T_{k,Z}= T_{k\vert Z\to Z}$ and $T_{k,Z'}= T_{k\vert Z'\to Z'}$. 
If $(T_{1,Z},\ldots, T_{d,Z})$ and $(T_{1,Z'},\ldots, T_{d,Z'})$ 
both admit isometric dilations (resp. polynomially
bounded isometric dilations, resp. polynomially
bounded isometric dilations on a UMD Banach space),
then the same holds true for $(T_1,\ldots,T_d)$.
\end{lemma}

\begin{proof}
By assumption, $(T_{1,Z},\ldots, T_{d,Z})$ admits an
isometric dilation $(U_1,\ldots, U_d)$ on 
some $Y$ and $(T_{1,Z'},\ldots, T_{d,Z'})$ admits an
isometric dilation $(V_1,\ldots, V_d)$ on 
some $Y'$. Define the direct sum operator
$W_k=U_k\oplus V_k\colon Y\overset{\infty}{\oplus} Y'\to
Y\overset{\infty}{\oplus} Y'$ for each
$k=1,\ldots, d$. Then $(W_1,\ldots, W_d)$ is an
isometric dilation of $(T_1,\ldots,T_d)$.
The entire statement follows.
\end{proof}

We now assume that $X$ is a UMD Banach space and that
$(T_1,\ldots,T_d)$ admits a bounded
$H^\infty$-functional calculus, and we aim at proving property (ii) of Theorem
\ref{1Main1}. As indicated in the previous sub-section,
the fact that every $T_k$ is an $R$-Ritt$_E$ operator follows from Lemma 
\ref{2Delta}.

By the Mean ergodic theorem for reflexive Banach spaces, for 
every Ritt$_E$ operator $T$ on $X$, we have a
direct sum decomposition,
\begin{equation}\label{4k}
X=\Bigl[\underset{j=1}{\overset{N}{\oplus}}{\rm Ker}
\left(I_X-\overline{\xi_j}T\right)\Bigr]
\oplus\overline{\rm Ran}\left(\prod_{j=1}^{N}
\left (I_X-\overline{\xi_j}T
\right)\right).
\end{equation}
This is stated in \cite[Lemma 3.4]{LMRMN} in the case
when $X$ is a Hilbert space and the proof given for it 
applies verbatim to all reflexive spaces.

We first apply the above decomposition for $T=T_1$, that is, we write
\begin{equation}\label{4k=1}
X=\Bigl[\underset{j=1}{\overset{N}{\oplus}}
{\rm Ker}\left(I-\overline{\xi_j}T_1\right)\Bigr]
\oplus\overline{\rm Ran}
\left(\prod_{j=1}^{N}
\left(I_X -\overline{\xi_j}T_1\right)\right).
\end{equation}
Then we consider $T_2$.
Each summand in the decomposition (\ref{4k=1}) is 
$T_2$-invariant and, of course, reflexive.
Hence, we can apply (\ref{4k}) with
$X$ replaced by any of the summands in (\ref{4k=1})
and with $T=T_2$ restricted to that summand. For instance, the first summand $Z_1={\rm Ker}\left(I-\overline{\xi_1}T_1\right)$
can be decomposed as
$$
Z_1 =
\Bigl[\underset{i=1}{\overset{N}{\oplus}}
{\rm Ker}(I_X-\overline{\xi_1}T_1)\cap 
{\rm Ker}(I_X-\overline{\xi_i}T_2)\Bigr]\oplus
\overline{\rm Ran}\left(\biggl(\underset{i=1}{\overset{N}{\prod}}(I_X-\overline{\xi_i} T_2)\biggr)_{\vert Z_1
\to Z_1}
\right).
$$
Applying Lemma \ref{4Invariant}, this reads
$$
Z_1 =
\Bigl[\underset{i=1}{\overset{N}{\oplus}}
{\rm Ker}(I_X-\overline{\xi_1}T_1)\cap 
{\rm Ker}(I_X-\overline{\xi_i}T_2)\Bigr]\oplus
\Bigl[{\rm Ker}\left(I-\overline{\xi_1}T_1\right)\cap
\overline{\rm Ran}\left(\underset{i=1}{\overset{N}{\prod}}(I_X-\overline{\xi_i} T_2)
\right)\Bigr].
$$
Applying this to all summands of in (\ref{4k=1}), we obtain
the decomposition
\begin{align*}
X& =\Bigl[\underset{i,j=1}{\overset{N}{\oplus}}
{\rm Ker}(I_X-\overline{\xi_j}T_1)\cap 
{\rm Ker}(I_X-\overline{\xi_i}T_2)\Bigr]\oplus
\Bigl[\underset{j=1}{\overset{N}{\oplus}}
{\rm Ker}(I_X -\overline{\xi_j}T_1)\cap
\overline{\rm Ran}\left(\underset{i=1}{\overset{N}{\prod}}(I_X-\overline{\xi_i} T_2)
\right)\Bigr] \\&\oplus \Bigl[
\underset{\jmath=1}{\overset{N}{\oplus}}\overline{\rm Ran}
\left(\underset{j=1}{\overset{N}{\prod}}(I_X-\overline{\xi_j}T_1)\right)\cap 
{\rm Ker}(I_X-\overline{\xi_i}T_2)\Bigr]
\oplus\overline{\rm Ran}
\left(\underset{j=1}{\overset{N}{\prod}}\underset{i=1}{\overset{N}{\prod}}
(I_X-\overline{\xi_j} T_1)(I_X-\overline{\xi_i} T_2)\right).
\end{align*}

For any $\Lambda\subset\{1,2,\cdots,d\}$, we define
$$
B_\Lambda = \underset{k\in\Lambda}{\prod}\, 
\underset{j=1}{\overset{N}{\prod}}
(I_X-\overline{\xi_j} T_k)
\qquad\hbox{and}\qquad
F_\Lambda =\overline{\rm Ran}(B_\Lambda),
$$
and we set $\Lambda^c=\{1,\ldots,d\}\setminus\Lambda$.
Then arguing as above, we obtain (by induction) the direct sum decomposition
\begin{equation}\label{4X}
X= \underset{\Lambda}{\oplus}\left(
\underset{j\in \{1,\ldots,N\}^{\Lambda^c}}{\oplus}
\left(\underset{k\in\Lambda^c}{\cap} {\rm Ker}(I_X -\overline{\xi_{j_k}} T_k)\right)
\cap F_\Lambda\right).
\end{equation}
Each summand in (\ref{4X})
is $T_k$-invariant for every $k=1,\ldots,d$. 
Hence by Lemma \ref{4DirectSumDilation}, it suffices to show that 
for every $\Lambda\subset\{1,2,\cdots,d\}$ and for every
$j\in \{1,\ldots,N\}^{\Lambda^c}$, the restriction of $(T_1,\ldots,T_d)$ 
to the space
$$
Z: = \left(\underset{k\in\Lambda^c}{\cap} {\rm Ker}(I_X -\overline{\xi_{j_k}} T_k)\right)
\cap F_\Lambda
$$
admits a polynomially bounded isometric dilation on some UMD Banach space.

To prove this, assume for simplicity
that $\Lambda=\{m+1,\ldots d\}$ for some $m\in\{0,\ldots,d\}$
(the general case can be proved similarly), so that
$$
Z = {\rm Ker}(I_X -\overline{\xi_{j_1}} T_1)
\cap \cdots\cap{\rm Ker}(I_X -\overline{\xi_{j_m}} T_m)
\cap \overline{\rm Ran}(B_\Lambda),
$$
for a certain $m$-tuple $(j_1,\ldots,j_m)$ of $\{1,\ldots,N\}$.
The case $m=d$ corresponds to $\Lambda=\emptyset$ and 
$B_\Lambda=I_X$.

It is plain that $\overline{\rm Ran}(B_\Lambda)\cap Z=Z$.
Moreover each summand in $(\ref{4X})$ is $B_\Lambda$-invariant.
Hence by Lemma \ref{4Invariant}, we have
$$
Z =\overline{\rm Ran}\bigl(B_{\Lambda\vert Z\to Z}\bigr) = 
\overline{\rm Ran}\left(
\underset{k=m+1}{\overset{d}{\prod}} 
\underset{i=1}{\overset{N}{\prod}} \left(I_X-\overline{\xi_i} 
T_{k,Z}\right)\right),
$$
where we denoted $T_{k,Z}$ for $T_{k\vert Z\to Z}$.

It follows from Lemma \ref{2Runge} that the sub-family $(T_{m+1},
\ldots,T_d)$ admits a bounded $H^\infty$- functional calculus
on $X$. Consequently, $(T_{m+1,Z},\ldots,T_{d,Z})$ 
admits a bounded $H^\infty$-functional calculus
on $Z$. 
Applying the special case treated in Sub-section \ref{4Special}, we
deduce that the $d$-tuple 
$(T_{m+1,Z},\ldots,T_{d,Z})$
admits a polynomially bounded isometric dilation 
$(U_{m+1},\ldots,U_d)$ on some UMD Banach space
$Y$. Thus, we have two bounded operators $J\colon Z\to Y$ and
$Q\colon Y\to Z$ such that
$$
T_{m+1}^{n_{m+1}}\cdots T_d^{n_d}= QU_{m+1}^{n_{m+1}}\cdots U_d^{n_d}J,
\qquad n_{m+1},\ldots,n_d\geq 0.
$$
For any $k=1,\ldots,m$, define $U_k=\xi_{j_k} I_Y\in B(Y)$.
Since $\vert \xi_{j_k}\vert =1$,
these operators are isometric isomorphisms and hence, 
$(U_1,\ldots,U_d)$ is a commuting $d$-tuple 
of isometric isomorphisms on $Y$. Clearly, this family 
is polynomially bounded.
Let $x\in Z$ and consider integers $n_{1},\ldots,n_d\geq 0.$
For every $k=1,\ldots m$, $Z\subset {\rm Ker}(I_X -\overline{\xi_{j_k}})$ 
hence $T_k(x) = \xi_{j_k}x$. Therefore, 
\begin{align*}
T_{1}^{n_{1}}\cdots T_d^{n_d}x & = T_{m+1}^{n_{m+1}}\cdots T_d^{n_d}\bigl(
T_{1}^{n_{1}}\cdots T_m^{n_m}x\bigr)\\
& = \xi_{j_1}^{n_1}\cdots \xi_{j_m}^{n_m}\, 
T_{m+1}^{n_{m+1}}\cdots T_d^{n_d}x
\\
& = \xi_{j_1}^{n_1}\cdots \xi_{j_m}^{n_m} \, QU_{m+1}^{n_{m+1}}\cdots U_d^{n_d}Jx\\
&=QU_{1}^{n_{1}}\cdots U_d^{n_d}Jx.
\end{align*}
This yields the dilation property.

\subsection{(iii) implies (i)}
This proof is an extension of the proof of \cite[Theorem 5.1]{LMRMN} 
to the Banach space case. We will use ingredients of the latter.
We assume (iii). That  $(T_1,\ldots,T_d)$ admits a polynomially 
bounded isometric dilation readily implies that $(T_1,\ldots,T_d)$ 
itself is polynomially bounded. Indeed, if $(T_1,\ldots,T_d)$
satisfies Definition \ref{2PB}, then $\varphi(T_1,\ldots,T_d)=
Q\varphi(U_1,\ldots,U_d)J$ for all $\varphi\in\P_d$, which implies
$$
\norm{\varphi(T_1,\ldots,T_d)}\leq\norm{Q}\norm{J} 
\norm{\varphi(U_1,\ldots,U_d)},\qquad\varphi\in\P_d.
$$
In turn, polynomial boundedness implies the existence of  
a  constant $K\geq 1$ such that
\begin{equation}\label{4PB}
\norm{\Phi(uT_1,\ldots,uT_d)}\leq K\norm{\Phi}_{\infty,{\mathbb D}^d},
\qquad \Phi\in H^\infty(\Ddb^d), u\in (0,1).
\end{equation}
The proof of this estimate is the same as the one of the implication
``$(ii)\Rightarrow(i)$" in Lemma \ref{2Runge}, so we omit it.

By assumption, every $T_k$ is $R$-Ritt$_E$. Hence, by \cite[Lemma 3.1]{B} and 
\cite[Remark 2.7]{BLM},
there exist $0<r<s<1$ such that $\sigma(T_k)\subset\overline{E_r}$ for 
every $k=1,\dots,d$, $\partial E_s\cap\overline{E_r}=E$, and the set
$$
\F_1 : = \Bigl\{\Bigl(\prod_{j=1}^N(\xi_j-z)\Bigr)R(z,uT_k)\, :\, k=1,\ldots,d,\,
u\in(0,1),\, z\in\partial E_s\setminus E\Bigr\}
$$
is $R$-bounded.

Our goal is to prove an estimate
\begin{equation}\label{4Goal}
\norm{\varphi(uT_1,\ldots,uT_d)}\lesssim 
\norm{\varphi}_{\infty, E_s^d},\qquad \varphi\in\P_d,\,
u\in (0,1),
\end{equation}
where $\lesssim$ stands for an inequality up to a constant that depends
neither on $\varphi$ or $u$. This estimate immediately extends to $u=1$,
hence by Lemma \ref{2Runge}, it implies the expected result that 
$(T_1,\ldots,T_d)$ admits a bounded $H^\infty$-functional calculus.
We now fix some $\varphi\in\P_d$.

Let $(\theta_i)_{i\geq 1}$, $(\phi_i)_{i\geq 1}$ and $(\psi_i)_{i\geq 1}$ 
be three sequences of $H^\infty(\Ddb)$ that 
satisfy the four properties (a), (b), (c) and (d)
of the proof of \cite[Theorem 5.1]{LMRMN}. 
Using the identity 
$$
1=\sum_{i=1}^\infty \theta_i(z)\phi_i(z)\psi_i(z),\qquad z\in\Ddb,
$$
given by 
property (d), we may write, for any $u\in (0,1)$,
$$
\varphi(uT_1, \ldots,uT_d) = \lim_n S(n,u), 
$$
where
$$
S(n,u) =  
\sum_{i_1,\ldots,i_d=1}^{n}\, 
\varphi(uT_1,\ldots,uT_d)\prod_{k=1}^d \theta_{i_k}(uT_k)
\prod_{k=1}^d \phi_{i_k}(uT_k)\prod_{k=1}^d \psi_{i_k}(uT_k).
$$

For any 
$i_1,\ldots,i_d\geq 1$ and any $u\in(0,1)$, we have
\begin{align*}
(2\pi i)^d  \,&\varphi(uT_1,\ldots,uT_d)\prod_{k=1}^d \theta_{i_k}(uT_k)\\
&=  \int_{(\partial E_s)^d} \varphi(z_1,\ldots,z_d)\prod_{k=1}^d 
\bigl(\theta_{i_k}(z_k) R(z_k,uT_k)\bigr)\,\prod_{k=1}^d dz_k\\
&=  \int_{(\partial E_s)^d} \varphi(z_1,\ldots,z_d)\prod_{k=1}^d 
\Bigl(\prod_{j=1}^N(\xi_j-z_k)
R(z_k,uT_k)\Bigr)\,\prod_{k=1}^d \Biggl(\frac{
\theta_{i_k}(z_k)}{\prod_{j=1}^N(\xi_j-z_k)}\Biggr)\,   
\prod_{k=1}^d dz_k.
\end{align*}
Hence, by \cite[Proposition 8.5.2]{HVVW2}, by property (c) in the
proof of \cite[Theorem 5.1]{LMRMN} and by
the $R$-boundedness 
of $\F_1$, the set 
$$
\F_2 :=\Bigl\{\varphi(uT_1,\ldots,uT_d)\prod_{k=1}^d \theta_{i_k}(uT_k)\,
:\, i_1,\ldots,i_d\geq 1,\, u\in(0,1)\Bigr\}
$$
is $R$-bounded, with an estimate
\begin{equation}\label{4F2}
\R\bigl(\F_2\bigr)\lesssim \norm{\varphi}_{\infty, E_s^d}.
\end{equation}

Let $x\in X$ and $y\in X^*$. By (\ref{2CS}), we have
$$
\bigl\vert \langle S(n,u)x,y\rangle\bigr\vert\leq A(x) B(y),
$$
where
$$
A(x) = \biggnorm{\sum_{i_1,\ldots,i_d=1}^{n}\varepsilon_{i_1,\ldots,i_d}
\otimes \varphi(uT_1,\ldots,uT_d)\prod_{k=1}^d \theta_{i_k}(uT_k)
\prod_{k=1}^d \phi_{i_k}(uT_k)x}_{{\rm Rad}({\mathbb N}^d;X)}
$$
and
$$
B(y) = \biggnorm{\sum_{i_1,\ldots,i_d=1}^{n}\varepsilon_{i_1,\ldots,i_d}
\otimes \prod_{k=1}^d \psi_{i_k}(uT_k)^* y}_{{\rm Rad}({\mathbb N}^d;X^*)}.
$$
The $R$-boundedness of the set $\F_2$ and the estimate (\ref{4F2}) imply that
$$
A(x)\lesssim \norm{\varphi}_{\infty, E_s^d}\, A'(x),
$$
with
$$
A'(x) = \biggnorm{\sum_{i_1,\ldots,i_d=1}^{n}\varepsilon_{i_1,\ldots,i_d}
\otimes 
\prod_{k=1}^d \phi_{i_k}(uT_k)x}_{{\rm Rad}({\mathbb N}^d;X)}.
$$
Equivalently,
$$
A'(x)^2 =\int_{{\mathcal X}_0}
\biggnorm{\sum_{i_1,\ldots,i_d=1}^{n}\varepsilon_{i_1,\ldots,i_d}(\lambda)
\prod_{k=1}^d \phi_{i_k}(uT_k)x}_X^2\, d\Pdb_0(\lambda).
$$
For any $\lambda\in \X_0$, we may write
$$
\sum_{i_1,\ldots,i_d=1}^{n}\varepsilon_{i_1,\ldots,i_d}(\lambda)
\prod_{k=1}^d \phi_{i_k}(uT_k)x = \Phi_\lambda(uT_1,\ldots,uT_k)x,
$$
where $\Phi_\lambda\colon \Ddb^d\to\Cdb$ is given by
$$
\Phi_\lambda(z_1,\ldots,z_k) = \sum_{i_1,\ldots,i_d=1}^{n}\varepsilon_{i_1,\ldots,i_d}(\lambda)
\prod_{k=1}^d \phi_{i_k}(z_k).
$$
It follows from property (a) in 
the proof of \cite[Theorem 5.1]{LMRMN} that this function is bounded on
$\Ddb^d$, and that 
$$
\sup\bigl\{\norm{\Phi_\lambda}_{\infty,{\mathbb D}^d}\, 
:\,\lambda\in\X_0\bigr\} <\infty.
$$
Therefore, it follows from (\ref{4PB}) that
$A'(x)\lesssim \norm{x}$. Thus, $A(x)\lesssim \norm{\varphi}_{\infty, E_s^d}\norm{x}$.

Note that since $(T_1,\ldots, T_d)$ is polynomially bounded,
the $d$-tuple $(T_1^*,\ldots, T_d^*)$ is also polynomially bounded. 
Moreover $\psi_{i_k}(uT_k)^* =\psi_{i_k}(uT_k^*)$ for all $k,i_k,u$.
Hence, the argument used to estimate $A'(x)$ also shows that 
$B(y)\lesssim \norm{y}$. We finally obtain an estimate 
$$
\bigl\vert \langle S(n,u)x,y\rangle\bigr\vert\lesssim 
\norm{\varphi}_{\infty, E_s^d}\norm{x}\norm{y}.
$$
Letting $n\to\infty$, we deduce 
the expected estimate (\ref{4Goal}).

\begin{remark}\label{4Converse}
Combined with Lemma \ref{2Runge}, the above proof 
actually shows that if $(T_1,\ldots,T_d)$ is a
commuting $d$-tuple of $R$-Ritt$_E$ operators on any Banach space $X$, then 
$(T_1,\ldots,T_d)$ admits a bounded $H^\infty$-functional calculus 
if and only if $(T_1,\ldots,T_d)$ is polynomially bounded. 
\end{remark}

\section{The case when $X$ has property $(\alpha)$}\label{5M2} 
The main objective of this section
is to prove Theorem \ref{1Main2}. 
This will be achieved after  establishing two intermediate
propositions of independent interest.

Let $(T_1,\ldots,T_d)$ be a commuting $d$-tuple of 
Ritt$_E$ operators on some Banach space $X$. It follows from 
Lemma \ref{2Runge} that 
if $(T_1,\ldots,T_d)$ admits a bounded $H^\infty$-functional calculus, 
then each $T_k$  admits a bounded $H^\infty$-functional calculus.
It is well known that the converse is already false for Ritt operators.
In Remark \ref{5CEX} below, we provide an example for the sake of completeness.
However, we have the following remarkable result.

\begin{proposition}\label{5Joint}
Assume that $X$ is a Banach lattice or 
$X$ has property $(\alpha)$. If  each $T_k$  
admits a bounded $H^\infty$-functional calculus, 
then the $d$-tuple $(T_1,\ldots,T_d)$ admits a bounded 
$H^\infty$-functional calculus.
\end{proposition}

\begin{proof}
This result is known for Ritt operators,
see \cite[Theorem 3.1]{ArLM1}. It is easy
to adapt the proof to Ritt$_E$ operators without much effort. 
Indeed, the proof
of \cite[Theorem 3.1]{ArLM1} relies on 
the Franks-McIntosh type decomposition and the factorization result
established in \cite[Theorem 6.1]{ArLM1} and \cite[Remark 6.3]{ArLM1},
respectively. 
Then this proof also works for 
Ritt$_E$ operators, using \cite[Proposition 4.6]{B}
instead of \cite[Theorem 6.1]{ArLM1} and 
\cite[Lemma 4.7]{B} instead of \cite[Remark 6.3]{ArLM1}.
We skip the details.
\end{proof}

\begin{remark}\label{5CEX}
Let $Y$ be a Banach space and let $X={\rm Rad}^2(Y)$.
We may define $T_1,T_2\in B(X)$ by setting
$$
T_1\Bigl(\sum_{k,j} \varepsilon_k\otimes \varepsilon_j\otimes x_{kj}\Bigr) 
= \sum_{k,j} (1-2^{-k})\varepsilon_k\otimes \varepsilon_j\otimes x_{kj}
$$
and
$$
T_2\Bigl(\sum_{k,j} \varepsilon_k\otimes \varepsilon_j\otimes x_{kj}\Bigr) 
= \sum_{k,j} (1-2^{-j})\varepsilon_k\otimes \varepsilon_j\otimes x_{kj},
$$
for all finitely supported families $(x_{kj})_{k,j\geq 1}$ of $Y$,
and then extending by continuity. Applying 
\cite[Proposition 6.1.13]{HVVW2}, it is easy to
check that $T_1,T_2$ are
Ritt operators with a bounded $H^\infty$-functional calculus. 
These two operators obviously commute. Furthermore, for any $\varphi\in H^\infty(\Ddb^2)$
and $u\in(0,1)$, 
$$
\varphi(uT_1,uT_2)\Bigl(\sum_{k,j} \varepsilon_k\otimes \varepsilon_j\otimes x_{kj}\Bigr) 
= \sum_{k,j} \varepsilon_k\otimes \varepsilon_j\otimes 
\varphi\bigl(u(1-2^{-k}),u(1-2^{-j})\bigr)
x_{kj},
$$
for all  finitely supported families $(x_{kj})_{k,j\geq 1}$ of $Y$.

By interpolation theory (see e.g. \cite[Chapter VII]{Ga}), for any bounded 
family $(z_{kj})_{k,j\geq 1}$ of complex numbers, there exists $\varphi\in H^\infty(\Ddb^2)$
such that $\varphi(1-2^{-k},1-2^{-j})=z_{kj}$ for all $k,j\geq 1$ and moreover we have an estimate
$\norm{\varphi}_{\infty,{\mathbb D}^2}\lesssim \sup\{\vert z_{kj}\vert\, : \, k,j\geq 1\}$.
We deduce that if $(T_1,T_2)$ is polynomially bounded, then
$Y$ has property $(\alpha)$. 

Thus, if $Y$ does not have property $(\alpha)$, we have constructed 
two commuting Ritt operators $T_1,T_2$ on $X$ such that 
$T_1$ and $T_2$ have a bounded $H^\infty$-functional calculus,
but $(T_1,T_2)$ does not have a bounded $H^\infty$-functional calculus.
\end{remark}

For any integer $l\geq 0$, consider the intervals
$$
I_l=(2^{-(l+1)}\pi, 2^{-l}\pi)\subset (0,\pi)
\qquad\hbox{and}\qquad
I'_l=(-2^{-l}\pi, -2^{-(l+1)}\pi)\subset (-\pi,0).
$$
Let $\H$ be the algebra of all holomorphic functions
defined on an open neighborhood of $\overline{\Ddb}$.
For any  $\varphi\in\H$ and for any $l\geq 0$, 
let ${\rm Var}(\varphi_{\vert I_l})$
be the variation of the function $t\mapsto \varphi(e^{it})$ on $I_l$.
We define ${\rm Var}(\varphi_{\vert I'_l})$ similarly.

For any Banach space $Y$, let ${\mathfrak s}_Y\colon 
\ell^2_{\mathbb Z}(Y)\to \ell^2_{\mathbb Z}(Y)$
be the shift operator defined by setting 
$$
{\mathfrak s}_Y\bigl((y_m)_{m\in{\mathbb Z}}\bigr)
= (y_{m-1})_{m\in{\mathbb Z}}, 
\qquad 
(y_m)_{m\in{\mathbb Z}} \in \ell^2_{\mathbb Z}(Y).
$$
For any $\varphi\in\H$, we
define
$$
\norm{\varphi}_Y : = \bignorm{\varphi({\mathfrak s}_Y)\colon \ell^2_{\mathbb Z}(Y)\longrightarrow\ell^2_{\mathbb Z}(Y)}.
$$
Since the operator 
$\varphi({\mathfrak s}_Y)$
is the Fourier multiplier associated with the symbol 
$\varphi_{\vert{\mathbb T}}$, the following Marcinkiewicz type theorem
is a special case of \cite[Theorem 4.3]{Blunck}. 

\begin{lemma}\label{5Marcin}
Assume that $Y$ is a UMD Banach space. There exists a constant $K\geq 1$
such that for all $\varphi\in\H$, we have
$$
\norm{\varphi}_Y
\leq K\bigl(\norm{\varphi}_{\infty,{\mathbb D}} 
+\sup_{l\geq 0}
{\rm Var}(\varphi_{\vert I_l}) + \sup_{l\geq 0}
{\rm Var}(\varphi_{\vert I'_l})\bigr)
$$
\end{lemma}

In the sequel, we use $H^\infty$-functional calculus for sectorial 
operators, for which we refer to \cite{Ha} or \cite[Chapter 10]{HVVW2}. 
For any
$\omega\in(0,\pi)$, we denote $\Sigma_\omega$ as the set
of all non-zero complex numbers $z$ such 
that $\vert{\rm Arg}(z)\vert<\omega$.

We will use the well known fact that if $T\in B(X)$ is 
a power bounded operator, then $I_X-T$ is sectorial 
(see e.g. \cite[Corollary 2.29]{LM-Book}).

\begin{proposition}\label{5Arhancet}
Let $X,Y$ be Banach spaces and assume that
$Y$ is UMD. Let $T\in B(X)$ be a power
bounded operator and assume that there exists
$K\geq 0$ such that 
$$
\norm{\varphi(T)}\leq K\norm{\varphi}_Y,\qquad \varphi\in\H.
$$
Then for any $\omega\in\bigl(\frac{\pi}{2},\pi\bigr)$, 
$I_X-T$ admits a bounded $H^\infty(\Sigma_\omega)$-functional calculus.
\end{proposition}

\begin{proof}
Set $A=I_X-T$.
We fix $\omega\in\bigl(\frac{\pi}{2},\pi\bigr)$ and we let
$\R_\omega\subset H^\infty(\Sigma_\omega)$ be the algebra 
of all rational functions with a non positive degree and
poles outside $\overline{\Sigma_\omega}$. 
Let $f\in \R_{\omega}$. We 
may define a function $\varphi\in \H$ by 
$\varphi(z) = f(1-z)$ and we have
$f(A)=\varphi(T)$. Therefore, the assumption yields an estimate 
$\norm{f(A)}\leq K\norm{\varphi}_Y$.

Since $Y$ is UMD, the calculation in the proof of 
\cite[Proposition 4.7]{ALM} and Lemma \ref{5Marcin} lead to 
an estimate
$\norm{\varphi}_Y\lesssim \norm{f}_{\infty,\Sigma_\omega}$.
Thus, we have an estimate
$$
\norm{f(A)}\lesssim \norm{f}_{\infty,\Sigma_\omega},\qquad f\in\R_\omega.
$$
By \cite[Proposition 2.10]{LM1}
(see also \cite[Sub-section 5.3.4]{Ha}), 
this implies that $A$ admits a bounded 
$H^\infty(\Sigma_\omega)$-functional calculus.
\end{proof}

\begin{proof}[Proof of Theorem \ref{1Main2}]
The implication ``$(i)\Rightarrow(ii)$" is given by Proposition \ref{5Joint}. 
The implication ``$(ii)\Rightarrow(iii)$" follows from  Theorem \ref{1Main1}
and its proof.
Indeed, if $X$ has property $(\alpha)$, then the Bochner space
$L^2(\X_1;X)$ has property $(\alpha)$ as well, by \cite[Proposition 7.5.3]{HVVW2}. 
Therefore,
$G(I;X)$ has property $(\alpha)$ for any countable set $I$. 
The implication ``$(iii)\Rightarrow(iv)$" is trivial. 

To prove ``$(iv)\Rightarrow(i)$", let $k\in\{1,\ldots,d\}$ and
set $T=T_k$ for convenience. 
By assumption, there
exists a UMD Banach space $Y$, an isometric isomorphism
$U\in B(Y)$ and two bounded maps $J\colon X\to Y$ and 
$Q\colon Y\to X$ such that $T^n=QU^nJ$ for all $n\geq 0$.
This implies $\varphi(T)=Q\varphi(U)J$, and hence 
$\norm{\varphi(T)}\lesssim\norm{\varphi(U)}$ for all $\varphi\in \P_1$.
By elementary approximation, this estimate holds as well
for any $\varphi\in\H$.
By a classical transference argument 
(see e.g. \cite[Proposition 8.7]{LM-Book}), the fact that 
$U$ is an isometric isomorphism implies that 
$\norm{\varphi(U)}\leq\norm{\varphi}_Y$ for all $\varphi\in \H$. 
Therefore, we have an estimate
$$
\norm{\varphi(T)}\lesssim\norm{\varphi}_Y,\qquad \varphi\in\H.
$$
For any $j=1,\ldots,N$, we may apply the above reasoning to 
$\overline{\xi_j}T$ instead of $T$. According to Proposition 
\ref{5Arhancet}, the above estimate therefore implies that 
$A_j: = I_X-\overline{\xi_j}T$ admits a bounded 
$H^\infty(\Sigma_\omega)$-functional calculus
for any $\omega\in\bigl(\frac{\pi}{2},\pi\bigr)$.

By assumption, $T=T_k$ is $R$-Ritt. This implies that
$\sigma(A_j)\subset \{z\in\Cdb\, :\, {\rm Re}(z)\geq 0\}$ and 
that the set
$$
\bigl\{zR(z,A_j)\, :\, {\rm Re}(z)<0\bigr\}
$$
is $R$-bounded. Applying \cite[Proposition 5.1]{KaW1}, we deduce 
that $A_j$ admits a bounded 
$H^\infty(\Sigma_{\omega_j})$-functional calculus
for some $\omega_j\in\bigl(0,\frac{\pi}{2}\bigr)$.
By \cite[Theorem 4.3]{BLM}, this implies that 
$T=T_k$ admits a bounded $H^\infty$-functional calculus.
\end{proof}

\bigskip
\noindent
{\bf Acknowledgements.} 
This work 
has been supported by the EIPHI Graduate school (contract ANR-17-EURE-0002).

\end{document}